\documentclass[a4paper, reqno, 14pt]{amsart}
\usepackage{amsthm,amsfonts,amssymb,amsmath,amsxtra}

\usepackage{amscd}

\usepackage[all]{xy}
\SelectTips{cm}{}
\usepackage{verbatim}

\usepackage[margin=1.25in]{geometry}

\usepackage{mathrsfs}

\RequirePackage{xspace}
\RequirePackage{etoolbox}
\RequirePackage{enumitem}
\RequirePackage{tensor}
\RequirePackage{mathtools}

\newcommand{\BA}{\ensuremath{\mathbb {A}}\xspace}

\newcommand{\BD}{\ensuremath{\mathbb {D}}\xspace}

\newcommand{\BM}{\ensuremath{\mathbb {M}}\xspace}
\newcommand{\BN}{\ensuremath{\mathbb {N}}\xspace}

\newcommand{\BP}{\ensuremath{\mathbb {P}}\xspace}
\newcommand{\BQ}{\ensuremath{\mathbb {Q}}\xspace}

\newcommand{\BV}{\ensuremath{\mathbb {V}}\xspace}

\newcommand{\BX}{\ensuremath{\mathbb {X}}\xspace}

\newcommand{\BZ}{\ensuremath{\mathbb {Z}}\xspace}

\newcommand{\Fb}{\boldsymbol{\mathcal F}}

\newcommand{\CE}{\ensuremath{\mathcal {E}}\xspace}
\newcommand{\CF}{\ensuremath{\mathcal {F}}\xspace}
\newcommand{\CG}{\ensuremath{\mathcal {G}}\xspace}

\newcommand{\CL}{\ensuremath{\mathcal {L}}\xspace}
\newcommand{\CM}{\ensuremath{\mathcal {M}}\xspace}
\newcommand{\CN}{\ensuremath{\mathcal {N}}\xspace}
\newcommand{\CO}{\ensuremath{\mathcal {O}}\xspace}
\newcommand{\CP}{\ensuremath{\mathcal {P}}\xspace}

\DeclareMathOperator{\Coker}{Coker}

\DeclareMathOperator{\End}{End}

\DeclareMathOperator{\Gal}{Gal}
\newcommand{\GL}{{\mathrm{GL}}}

\DeclareMathOperator{\Hom}{Hom}

\newcommand{\id}{\ensuremath{\mathrm{id}}\xspace}
\let\Im\relax
\newcommand{\ideal}{\mathfrak}
\DeclareMathOperator{\Im}{Im}

\DeclareMathOperator{\Ker}{Ker}
\DeclareMathOperator{\length}{length}  
\DeclareMathOperator{\Lie}{Lie}

\DeclareMathOperator{\Nrd}{Nrd}

\DeclareMathOperator{\ord}{ord}

\DeclareMathOperator{\rank}{rank}

\DeclareMathOperator{\Res}{Res}

\DeclareMathOperator{\Spec}{Spec}
\DeclareMathOperator{\Spf}{Spf}

\newcommand{\SuetO}{O}
\newcommand{\suetO}{O}

\newcommand{\wt}{\widetilde}
\newcommand{\wh}{\widehat}

\newcommand{\ov}{\overline}

\newcommand{\lra}{\longrightarrow}

\newtheorem{theorem}{Theorem}[section]
\newtheorem{proposition}[theorem]{Proposition}
\newtheorem{lemma}[theorem]{Lemma}
\newtheorem {conjecture}[theorem]{Conjecture}
\newtheorem{corollary}[theorem]{Corollary}

\theoremstyle{definition}
\newtheorem{definition}[theorem]{Definition}

\newtheorem{remark}[theorem]{Remark}
\newtheorem{remarks}[theorem]{Remarks}

\numberwithin{equation}{section}

\newenvironment{altenumerate}
   {\begin{list}
      {\textup{(\theenumi)} }
      {\usecounter{enumi}
       \setlength{\labelwidth}{0pt}
       \setlength{\labelsep}{0pt}
       \setlength{\leftmargin}{0pt}
       \setlength{\itemsep}{\the\smallskipamount}
       \renewcommand{\theenumi}{\roman{enumi}}
      }}
   {\end{list}}

\numberwithin{equation}{section}

\setitemize[0]{leftmargin=*,itemsep=\the\smallskipamount}
\setenumerate[0]{leftmargin=*,itemsep=\the\smallskipamount}

\renewcommand{\to}{%
   \ifbool{@display}{\longrightarrow}{\rightarrow}%
   }
\let\shortmapsto\mapsto
\renewcommand{\mapsto}{%
   \ifbool{@display}{\longmapsto}{\shortmapsto}%
   }
\newlength{\olen}
\newlength{\ulen}
\newlength{\xlen}
\newcommand{\xra}[2][]{%
   \ifbool{@display}%
      {\settowidth{\olen}{$\overset{#2}{\longrightarrow}$}%
       \settowidth{\ulen}{$\underset{#1}{\longrightarrow}$}%
       \settowidth{\xlen}{$\xrightarrow[#1]{#2}$}%
       \ifdimgreater{\olen}{\xlen}%
          {\underset{#1}{\overset{#2}{\longrightarrow}}}%
          {\ifdimgreater{\ulen}{\xlen}%
             {\underset{#1}{\overset{#2}{\longrightarrow}}}
             {\xrightarrow[#1]{#2}}}}%
      {\xrightarrow[#1]{#2}}
   }
\makeatother
\newcommand{\xyra}[2][]{%
   \settowidth{\xlen}{$\xrightarrow[#1]{#2}$}%
   \ifbool{@display}%
      {\settowidth{\olen}{$\overset{#2}{\longrightarrow}$}%
       \settowidth{\ulen}{$\underset{#1}{\longrightarrow}$}%
       \ifdimgreater{\olen}{\xlen}%
          {\mathrel{\xymatrix@M=.12ex@C=3.2ex{\ar[r]^-{#2}_-{#1} &}}}%
          {\ifdimgreater{\ulen}{\xlen}%
             {\mathrel{\xymatrix@M=.12ex@C=3.2ex{\ar[r]^-{#2}_-{#1} &}}}
             {\mathrel{\xymatrix@M=.12ex@C=\the\xlen{\ar[r]^-{#2}_-{#1} &}}}}}%
      {\mathrel{\xymatrix@M=.12ex@C=\the\xlen{\ar[r]^-{#2}_-{#1} &}}}%
   }
\makeatletter
\newcommand{\xla}[2][]{%
   \ifbool{@display}%
      {\settowidth{\olen}{$\overset{#2}{\longleftarrow}$}%
       \settowidth{\ulen}{$\underset{#1}{\longleftarrow}$}%
       \settowidth{\xlen}{$\xleftarrow[#1]{#2}$}%
       \ifdimgreater{\olen}{\xlen}%
          {\underset{#1}{\overset{#2}{\longleftarrow}}}%
          {\ifdimgreater{\ulen}{\xlen}%
             {\underset{#1}{\overset{#2}{\longleftarrow}}}
             {\xleftarrow[#1]{#2}}}}%
      {\xleftarrow[#1]{#2}}
   }
\newcommand{\isoarrow}{%
   \ifbool{@display}{\overset{\sim}{\longrightarrow}}{\xrightarrow\sim}%
   }
   
\newsavebox{\lineone}
\newsavebox{\linetwo}
\newsavebox{\linethree}
\newlength{\lineonelen}
\newlength{\linetwolen}
\newlength{\linethreelen}
\newlength{\biggerlen}
\newcommand{\twolinestight}[2]{%
   \sbox{\lineone}{#1}%
   \sbox{\linetwo}{#2}%
   \settowidth{\lineonelen}{\usebox{\lineone}}%
   \settowidth{\linetwolen}{\usebox{\linetwo}}%
   \setlength{\biggerlen}{\maxof{\lineonelen}{\linetwolen}}%
   \begin{minipage}{\the\biggerlen}%
      \centering
      \usebox\lineone\\%
      \usebox\linetwo%
   \end{minipage}%
   }

\begin{document}

%\thanks{Research of Rapoport and Terstiege partially supported by SFB/TR 45 ``Periods, Moduli Spaces and
%Arithmetic of Algebraic Varieties" of the DFG. Research of Zhang partially supported by NSF grant DMS 1204365.}

%\thanks{Research of Rapoport and Terstiege partially supported by SFB/TR 45 ``Periods, Moduli Spaces and 
%Arithmetic of Algebraic Varieties" of the DFG. Research of Zhang partially supported by NSF grant DMS 1204365.}

\title{On the Drinfeld moduli problem of $p$-divisible groups}
\author{M. Rapoport}
\author{Th. Zink}

\date{\today}
\maketitle

\tableofcontents
\section{Introduction}\label{s:intro}

Let $F$ be a finite extension of $\BQ_p$, with ring of integers $O_F$
and uniformizer $\pi$. In \cite{D}, Drinfeld defines a certain moduli
problem of $p$-divisible groups.  Let us recall Drinfeld's theorem.

A {\it formal $O_F$-module}  over a scheme $S$ such that $p$ is locally nilpotent in $\CO_S$ is a  $p$-divisible group
$X$ with an action of $O_F$,
\begin{displaymath}
   \iota: O_F \longrightarrow \End X\, .
\end{displaymath} 
 If $X$ is defined over an $O_F$-scheme $S$, the
Lie algebra $\Lie X$ is naturally an $O_F$-module. If this coincides
with the $O_F$-module structure given by $d\iota$ we call $X$ a {\it strict formal $O_F$-module}.

Let $D$ be a central division algebra of invariant $1/n$ over $F$,
with maximal order $O_D$. Drinfeld defines a   {\it  special formal
  $O_D$-module} over an $O_F$-scheme $S$ to be a strict formal $O_F$-module
of height $n^2 [F:\BQ_p]$ equipped with an action $\iota\colon
O_D\to\End(X)$ of $O_D$ that extends the action of $O_F$ and such that, in each geometric
point of $S$, the eigenspaces of $\Lie X$ under 
the action of an unramified extension $O_{\tilde{F}}$ of $O_F$ of
degree $n$ contained in $O_D$  are  all one-dimensional.

It is easy to see that if $S=\Spec \bar k$ is the spectrum of the
algebraic closure of the residue field of $O_F$,  there is a unique
special formal $O_D$-module $\BX$ over $S$, up to $O_D$-linear
isogeny.  

Let $O_{\breve F}$ be the ring of integers in the completion of the maximal unramified extension $\breve F$ of $F$. Drinfeld defines as follows a  set-valued functor $\CM$ on the category ${\rm Nilp}_{O_{\breve F}}$ of $O_{\breve F}$-schemes $S$ such that the ideal sheaf $\pi\CO_S$ is locally nilpotent: the functor associates to $S\in {\rm Nilp}_{O_{\breve F}}$ the set of isomorphism classes of triples $(X, \iota, \rho)$, where $(X, \iota)$ is a special formal $O_D$-module over $S$ and where $\rho\colon X\times_S\bar S\to \BX\times_{\Spec \bar k}\bar S$ is a $O_D$-linear quasi-isogeny of height $0$. Here $\bar S=S\times_{\Spec O_{\breve F}}\Spec \bar k$. Drinfeld's theorem is that this functor is representable by a very specific formal scheme, namely,
 \begin{equation}\label{Urthm}
 \CM\simeq \hat\Omega^n_F\times_{\Spf O_F}\Spf O_{\breve F} .
 \end{equation}
 Here $\hat\Omega^n_F$ is the formal scheme over $\Spf O_F$ defined by Deligne, Drinfeld and Mumford, cf. \cite{D}. It has as generic fiber (associated rigid-analytic space) Drinfeld's $p$-adic halfspace associated to $F$,
 $$
 (\hat\Omega^n_F)^{\rm rig}=\BP^{n-1}_F\setminus \bigcup\nolimits_{H/F} H . 
 $$
Here $H$ ranges over the hyperplanes of $\BP_F^{n-1}$ defined over $F$. 

Drinfeld's  theorem has many applications, in particular to the  $p$-adic uniformization of Shimura varieties, cf. \cite{RZ}. It also has applications to arithmetic, e.g. \cite{Ne, Ri}. 

The generic fiber of Drinfeld's formal moduli space admits a tower of finite \'etale coverings (via level structures on the Tate module of the universal $p$-divisible group). As such it is a prominent example of a  \emph{local Shimura variety}, cf. \cite{RV}. 

Let $G$ be a reductive group over the local field $F$, let $b$ be an element of $ G(\breve F)$ and let  $\{\mu\}$ be a conjugacy class of minuscule cocharacters  of $G_{\ov F}$. 
 One requires that the $\sigma$-conjugacy class $[b]$ of $b$ is \emph{neutral acceptable}, i.e., $[b]\in B(G, \{\mu\})$, cf. \cite{K}. The triple $(G, b, \{\mu\})$ is called a \emph{local Shimura datum} over $F$, cf. \cite{RV}. 
One expects to be able to attach to these data a local Shimura variety which satisfies obvious functorial properties and more. This is a tower of rigid-analytic spaces $\BM(G, b, \{\mu\})=\{\BM^K\mid K\subset G(F) \}$ indexed by the open compact subgroups of $G(F)$,  defined over the \emph{reflex field}  $E=E(G, \{\mu\})$, cf. \cite{RV}. By \cite{RZ}, local Shimura varieties exist in many cases. In the PEL case, local Shimura varieties are related to Shimura varieties by \emph{non-archimedean uniformization}, cf. \cite[Thm. 6.36]{RZ}. One also expects to have \emph{integral models} over $O_E$ of $\BM^K$,  for judicious choices of the open compact subgroup $K$, i.e., formal schemes over $\Spf O_{\breve E}$ with Weil descent datum to $\Spf O_E$ whose associated rigid-analytic space is $\BM^K$.

 In Drinfeld's case $G=D^\times$, (the linear algebraic group over $F$ associated to) the multiplicative group of $D$. The cocharacter  $\{\mu\}$ of $G_{\ov F}\simeq \GL_{n, \ov F}$ is $(1,0,\ldots,0)$ and $b$ is a representative of the unique element of $B(G, \{\mu\})$. Drinfeld's theorem says not only that the member $\BM^K$, for the level subgroup $K=O_D^\times$, is the Drinfeld $p$-adic upper half space attached to $F$, but also that the moduli problem $\CM$ defines an integral model of $\BM^K$, which is even a $\pi$-adic formal scheme with semi-stable reduction.

In the theory of local Shimura varieties, the following question arises. Assume that $b$ is a representative of the unique \emph{basic} element of $B(G, \{\mu\})$.  Let $\varepsilon$ be a central cocharacter of $G$ defined over $F$, and set $\{\mu'\}=\{\mu\varepsilon\}$. Then $E(G, \{\mu\})=E(G, \{\mu'\})$.  Let $b'$ be a representative of the unique basic element of $B(G, \{\mu'\})$. The question is whether  the local Shimura varieties  $\BM(G, b, \{\mu\})$  and $\BM(G, b', \{\mu'\})$  are Galois twists of each other (unramified twists if the connected center of $G$ splits over an unramified extension). 

 In a similar vein, start with a local Shimura datum $(G, b, \{\mu\})$ over $F$ such that $b$ is basic.
Let $G' = \Res_{F/\mathbb{Q}_p}(G)$. Then $G' \otimes_{\mathbb{Q}_p}
\bar{\mathbb{Q}}_p$ is a product of $G \otimes_{F}
\bar{\mathbb{Q}}_p$ indexed by the embeddings of $F$ in
$\bar{\mathbb{Q}}_p$. We fix an embedding by choosing an 
isomorphism $\bar{F} \cong \bar{\mathbb{Q}}_p$. Define the conjugacy class 
$\{\mu_0\}$ of $G'$ to be  $\{\mu\}$ for this fixed embedding and to be trivial
for all other embeddings. Then $E(G', \{\mu_0\}) \subset E(G,\{\mu\})$ 
with respect to the chosen isomorphism. Let $\varepsilon$ be a central
cocharacter of $G'$ defined over $F \subset
\bar{\mathbb{Q}}_p$. We set $\{\mu'\} = \{\mu_0 \varepsilon\}$. We still have $E(G', \{\mu'\}) \subset
E(G,\{\mu\})$.  Let $b'$ be a
representative of the unique basic element in $B(G', \{\mu'\})$. 
The question is whether  the local Shimura varieties  
$\mathbb{M}(G, b, \{\mu\})$ and 
$\mathbb{M}(G', b', \{\mu'\})\otimes_{E(G', \{\mu'\})} E(G, \{\mu\})$  are
Galois twists of each other (unramified twists if the connected center
of $G'$ splits over an unramified extension of $F$).

The goal of  this paper is to prove that this last question has an affirmative answer for the Drinfeld datum $(G, b, \{\mu\})$. But, even better, we construct integral models for the members corresponding to the natural \emph{maximal} level subgroups of both local Shimura varieties and show that they are isomorphic, at least when $F/\BQ_p$ is unramified\footnote{In his talk 14 July 2016 in the Bonn Arbeitsgemeinschaft Arithmetische Geometrie, P.~Scholze explained his proof of Conjecture \ref{conjDr} below, i.e., how to remove the unramifiedness hypothesis. His proof is based on  our Theorem \ref{mainDR}\label{footnoteScholze}, but uses in addition the \emph{integral $p$-adic Hodge theory} of B.~Bhatt, M.~Morrow and P.~Scholze (arXiv:1602.03148), and more. Scholze will publish his proof elsewhere. }. These integral models are constructed by posing a moduli problem of $p$-divisible groups. This is substantially different from Drinfeld's moduli problem (unless $F=\BQ_p$), which is a moduli problem of strict formal $O_F$-modules. This contrast between \emph{relative} and \emph{absolute} Rapoport-Zink spaces is important also in other contexts: in the work of A.~Mihatsch on the Arithmetic Fundamental Lemma \cite{Mi} and in joint work of us with S.~Kudla \cite{KRZ} on $p$-adic uniformization of Shimura curves.  In fact, the approach  in \cite{KRZ} is modelled on the present paper, but involves in addition a polarization.

There is another well-known local Shimura datum $(G, b, \{\mu\})$, referred to as the \emph{Lubin-Tate datum}. Here $G=\GL_n$, $\{\mu\}=(1, 0,\ldots, 0)$, and $[b]\in B(G, \{\mu\})$ is the unique basic element. In this case, the corresponding local Shimura variety again has an explicit integral model  for its member $\BM^K$, where $K=\GL_n(O_F)$. For this local Shimura variety, we also give a positive answer to the question raised above, again in the strong form pertaining to integral models. This theorem is applied in the work of B.~Smithling, W.~Zhang and the second author on the Arithmetic Gan-Gross-Prasad conjecture \cite{RSZ3}.  

It should be pointed out that the two cases of integral models of $\BM^K$ considered here are essentially the only ones   \emph{known explicitly}, which justifies singling out these special cases of a general problem.

The lay-out of the paper is as follows. In section \ref{s:formu} we
formulate our  moduli functor and  state our main
results. In section \ref{s:kotteis}, we discuss the conditions on the
Lie algebras in the formulation of the moduli problem. In section
\ref{s:formalmod} we establish an isomorphism between our moduli
functor when restricted to $\bar k$-schemes with the original Drinfeld
moduli functor. The main tool here is the theory of displays. In
section \ref{s:localmod} we determine the local structure of our
moduli scheme.  Here the main tool is the theory of local models of Rapoport-Zink spaces. In section \ref{s:genfiber} we prove the compatibility
theorem with the Drinfeld moduli functor   in the generic fiber. This
proof is due to P.~Scholze, and uses his theory of $p$-divisible
groups over $O_C$, cf. \cite{SW}.  In section \ref{s:unramified} we
prove our {\it integral} representability conjecture in the case where
$F/\BQ_p$ is unramified. The proof uses the theory of relative
displays of T.~Ahsendorf. In the final section \ref{s:LT} we give the
Lubin-Tate variant of our main theorem.  

\smallskip

{\bf Acknowledgements} We thank P.~Scholze for explaining to us his proof of point (ii) in Theorem \ref{mainDR}, and for a critical reading of a first version of this paper. We also thank U.~G\"ortz for helpful discussions.

\section{Formulation of the main results}\label{s:formu}

Let $F$ be a field extension of degree $d$ of $\mathbb{Q}_p$. 
We denote by $\SuetO_F$ the ring of integers and by $\kappa$
the residue  field. We write $f = [\kappa: \mathbb{F}_p]$ for the
inertia index and $d = ef$. 

Fix $n\geq 2$. 
Let
$\Phi = \Hom_{\mathbb{Q}_p}(F, \bar{\mathbb{Q}}_p)$ be the set of
field embeddings. 
We fix an embedding $\varphi_0 : F \rightarrow   \bar{\mathbb{Q}}_p$. 
Let $r: \Phi \rightarrow
\mathbb{Z}$ be a function such that
\begin{equation}\label{propr}
r_{\varphi} = 
\begin{cases}
\begin{array}{ll}
1, &  \text{if} \; \varphi = \varphi_0\\
0 \; \text{or} \; n, & \text{if} \; \varphi \neq \varphi_0.
\end{array}
\end{cases}
\end{equation} 
The reflex field $E \subset \bar{\mathbb{Q}}_p$ of $r$ is characterized by 
\begin{displaymath}
\Gal(\bar{\mathbb{Q}}_p/E) = \{\sigma \in
\Gal(\bar{\mathbb{Q}}_p/\mathbb{Q}_p)  \; | \; r_{\sigma \varphi} =
r_{\varphi} \; \text{for all} \; \varphi \}. 
\end{displaymath}
We have $\varphi_{0}(F) \subset E$, and we consider $E$ as a field extension of $F$ via $\varphi_0$.

Let $D$ be a central division algebra of invariant $1/n$
over $F$. We  will   consider $p$-divisible groups $X$ of height $n^2 d$ over $O_E$-schemes $S$, with an action
of the ring of integers $O_D$ in $D$,
\begin{equation*}
 \iota : O_D \lra \End (X).
\end{equation*}
We will need to impose conditions on the induced action of $O_D$ on $\Lie X$. The first condition is the {\it Kottwitz condition} 
\begin{equation}\label{charpolDr}
 {\rm char} (\iota (x) \vert \Lie X) = \prod\nolimits_{\varphi} \, \,\varphi\big( {\rm chard} \, \, (x) (T)\big)^{r_\varphi} , 
\quad \forall x \in O_D\, .
\end{equation}
Here ${\rm chard} (x)$ denotes the {\it reduced characteristic polynomial} of $x$, a polynomial of degree $n$ with coefficients in $O_F$. The RHS is a polynomial in $O_E[T]$. It becomes a polynomial with coefficients in $\CO_S$ via  the structure morphism.

As we will show (cf. Proposition \ref{unram}), the Kottwitz  condition is all we need to yield a good moduli problem when $F/\BQ_p$ is unramified. When $F/\BQ_p$ is ramified,  the Kottwitz condition is too weak. To state the additional condition we impose, we need some preparation. 

Let $F^t$ be the maximal unramified subfield of $F$. We will write
$\Psi = \Hom_{\mathbb{Q}_p}(F^{t},\bar{\mathbb{Q}}_{p})$ for 
the set of field embeddings. Let $\psi_0={\varphi_0}_{|F^{t}}$. For an
embedding $\psi: F^{t} 
\rightarrow \bar{\mathbb{Q}}_p$ we set
\begin{equation}\label{abpsi1e}
\begin{aligned}
   A_\psi &= \{\varphi: F \rightarrow \bar{\mathbb{Q}}_p \;|\;
   \varphi_{|F^{t}} = \psi, \; \text{and} \; r_{\varphi} = n  \}\\
B_{\psi} &= \{\varphi: F \rightarrow \bar{\mathbb{Q}}_p \;|\;
   \varphi_{|F^{t}} = \psi, \; \text{and} \; r_{\varphi} = 0  \}.
\end{aligned}
\end{equation} 
Also, let $a_\psi=|A_\psi|$ and $b_\psi=|B_\psi|$.  For any
$O_E$-scheme $S$ we have a decomposition of
$O_{F^t}\otimes_{\BZ_p}\CO_S$-modules 
\begin{equation}\label{unrdecomp}
O_{F^t}\otimes_{\BZ_p}\CO_S=\bigoplus\nolimits_{\psi\in\Psi}\CO_S\, ,
\end{equation}
where the action of $\CO_{F^t}$ on the $\psi$-th factor is via
$\psi$. Hence for $(X, \iota)$ over $S$, we obtain a decomposition
into locally free $\CO_S$-modules,  
\begin{equation}\label{eigensp}
\Lie X=\bigoplus\nolimits_{\psi\in\Psi} \Lie\!_\psi X\,.
\end{equation}
The rank of $\Lie\!_\psi X$ is given by \eqref{charpolDr} as 
\begin{equation}
\rank\, \Lie\!_\psi X=a_\psi n^2\,  +\epsilon_{ \psi}\, n \,,
\end{equation}
where $\epsilon_{ \psi}$ is equal to $1$ if $\psi=\psi_0$, and is equal to $0$ if $\psi\neq \psi_0$.

Let $\pi$ be a uniformizer in $O_F$. Consider the  Eisenstein
polynomial $Q(T)$ of $\pi$ in $O_{F^t}[T]$. We consider the image
$Q_\psi(T)$ of $Q(T)$ in $\bar\BQ_p[T]$ under $\psi$, for $\psi\in\Psi$. In
$\bar\BQ_p[T]$ this has a decomposition into linear factors,  
\begin{equation}\label{decoveralgcl}
Q_\psi(T)=\prod_{\{\varphi\mid\,\, \varphi _{|F^t}=\psi\}} (T-\varphi(\pi)) . 
\end{equation}
Note that $\Gal(\bar\BQ_p/E)$ acts on the index set of this product, as is clear since the LHS is a polynomial in $O_E[T]$. 
We therefore obtain a decomposition in $O_E[T]$ 
\begin{equation}\label{Qpsi1e}
\begin{aligned}
Q_{\psi_0}(T)&=Q_0(T)\cdot Q_{A_{\psi_0}}(T)\cdot Q_{B_{\psi_0}}(T) ,  &\text{resp.}\\
Q_\psi(T)&= Q_{A_\psi}(T)\cdot Q_{B_\psi}(T) ,  &\text{ for $\psi\neq\psi_0$} .
\end{aligned}
\end{equation}
Here
$$Q_0(T)=T-\varphi_0(\pi), \quad Q_{A_\psi}(T)=\prod_{\varphi\in
  A_\psi}(T-\varphi(\pi)), \quad Q_{B_\psi}(T)=\prod_{\varphi\in
  B_\psi}(T-\varphi(\pi)). 
$$ 
Indeed, the action of $\Gal(\bar\BQ_p/E)$ stabilizes the
corresponding subsets in the index set on the RHS of
\eqref{decoveralgcl}. Now using the structure morphism $O_E\to \CO_S$,
we obtain an endomorphism $Q_{A_\psi}(\iota(\pi))$  of the
$\CO_S$-module $\Lie_\psi X$ that we denote by $ Q_{A_\psi}(\iota(\pi)
\vert\Lie_\psi X)$. We similarly define $ Q_{B_\psi}(\iota(\pi)
\vert\Lie_\psi X)$ and $Q_0(\iota(\pi)\vert \Lie_{\psi_0}
X)=\iota(\pi)\vert \Lie_{\psi_0} X-\varphi_0(\pi){\rm
  Id}_{\Lie_{\psi_0} X}$.  

The additional conditions we impose, that we call the {\it Eisenstein conditions},   are now the following identities of endomorphisms 
\begin{equation}\label{furtherDr2}
\begin{aligned}
   \big( Q_0\cdot  Q_{A_{\psi_0}}\big)(\iota(\pi) \vert\Lie_{\psi_0} X) & = 0, \\
    \bigwedge^{n+1} \big( Q_{A_{\psi_0}}(\iota(\pi) \vert\Lie_{\psi_0}
    X) \big)& =  0, \\ 
Q_{A_\psi}(\iota(\pi) \vert\Lie\!_\psi X) &   = 0 , \, \forall \psi\neq \psi_0 .
\end{aligned}
\end{equation}
 
\begin{remark}
  We note that the Eisenstein conditions only depend on the
  restriction of the $O_D$-action to $O_F$.  It will follow a
  posteriori from Corollary \ref{indepofpi} (flatness) that the moduli
  problem formulated using the Eisenstein conditions is independent of
  the choice of the uniformizer $\pi$.
\end{remark}
We first note the following statement.
\begin{proposition} \label{unram} If $F/\BQ_p$ is unramified, the
  Eisenstein conditions are implied by the Kottwitz condition.
\end{proposition}
\begin{proof}
  When $F=F^t$ is unramified over $\BQ_p$, the uniformizer $\pi$ lies
  in $F^t$ and $Q(T)=T-\pi$ is a linear polynomial. Furthermore,
  $A_\psi$ has at most one element for $\psi\neq \psi_0$, and
  $A_{\psi_0}=\emptyset$. Let $\psi\neq \psi_0$. If
  $A_\psi=\emptyset$, then $\Lie_\psi X=(0)$ and the Eisenstein
  condition relative to the index $\psi$ is empty; if $A_\psi$ has one
  element, the Eisenstein condition relative to the index $\psi$ is
  just equivalent to the definition of the $\psi$-th eigenspace in the
  decomposition \eqref{eigensp}. Something analogous applies to the
  index $\psi_0$.
\end{proof}
The following statement shows that the moduli problems considered in
this paper are indeed generalizations of Drinfeld's moduli problem.
We call the {\it Drinfeld function} the function $r^\circ$ with
$r^\circ_\varphi=0, \forall \varphi\neq\varphi_0$. In this case
$E_{r^\circ}=F$.
\begin{proposition}\label{comprDr}
  Assume that $r=r^\circ$. Then a p-divisible group $(X, \iota)$ as
  above, i.e., satisfying the Kottwitz condition \eqref{charpolDr} and
  the Eisenstein conditions \eqref{furtherDr2}, is a special formal
  $O_D$-module in the sense of Drinfeld \cite{D}.
\end{proposition}
\begin{proof} In this case, $A_\psi=\emptyset, \forall \psi$. Hence
  $\Lie_\psi X=(0)$ for $\psi\neq \psi_0$, and $\Lie_{\psi_0} X$ is a
  locally free $\CO_S$-module of rank $n$. Also, the endomorphism
  $Q_{A_{\psi_0}}(\iota(\pi) \vert\Lie_{\psi_0} X)$ is the identity
  automorphism. Hence the first Eisenstein condition implies that
  $Q_0(\iota(\pi) \vert\Lie_{\psi_0} X)=0$. Since
  $Q_0(T)=T-\varphi_0(\pi)$, it follows that $\iota(\pi)$ acts on
  $\Lie_{\psi_0} X$ through the structure morphism
  $O_E\to\CO_S$. The same is true for all elements of
  $O_{F^t}$. Hence $X$ is a strict formal $O_F$-module. Now the Kottwitz condition \eqref{charpolDr} tells us that
  the action of $O_D$ on $X$ is {\it special}, which proves
  the claim.
\end{proof}
\begin{definition}
  Fix a function $r:\Phi \to \BZ_{\geq 0}$, with corresponding reflex
  field $E=E_r$. A $p$-divisible group $X$ with action $\iota$ by
  $O_D$ over a $O_E$-scheme $S$ is called an
  \emph{$r$-special $O_D$-module}, if $X$ is of height $n^2d$ and $(X,
  \iota)$ satisfies the Kottwitz condition and the Eisenstein
  conditions relative to $r$.
\end{definition}
Hence the previous proposition shows that a $r^\circ$-special formal
$O_D$-module is just a special formal $O_D$-module
in the sense of Drinfeld \cite{D}.

For the formulation of the moduli problem we will make use of the following lemma. The lemma follows from section \ref{s:formalmod}, more precisely, Corollary \ref{uniupto}. Alternatively, the lemma   follows from the fact that $B(G, \{\mu\})$ has only one element, cf. \cite{K}, \S 6. Here $G=\Res_{F/\BQ_p}(D^\times)$ is the linear algebraic group over $\BQ_p$ associated to $D^\times$, and $\{\mu\}$ is the conjugacy class of cocharacters with component $(1, 0^{(n-1)})$ for $\varphi_0$ and component $(1^{(n)})$, resp. $(0^{(n)})$ for $\varphi\neq\varphi_0$, depending on whether $r_{\varphi}=n$ or $r_\varphi=0$. 
\begin{lemma}\label{uniqueframe} 
Fix $r$. Let $\bar k$ be an algebraic closure  of the residue field $\kappa_E$ of $O_E$. Any two  $r$-special $p$-divisible groups  over $\bar k$ are isogenous by a $O_D$-linear isogeny (which may be taken to be of height $0$). \qed
\end{lemma}

Now fix such a pair $(\BX, \iota_{\BX})$ over $\bar k$.
Denote by $O_{\breve E}$ the ring of integers in the completion of the maximal unramified extension of $E$. Then $\bar k$ is the residue field of $O_{\breve E}$. We consider the following set-valued functor $\CM_r$ on ${\rm Nilp}_{O_{\breve E}}$. It associates to $S\in {\rm Nilp}_{O_{\breve E}}$ the set of isomorphism
classes of triples $(X, \iota, \varrho)$, where $(X, \iota)$ is an $r$-special $O_D$-module over $S$, and where 
\begin{equation}
 \varrho : X \times_S \bar{S} \lra \BX \times_{\Spec \bar k} \bar{S}
\end{equation}
is a $O_D$-linear quasi-isogeny of height  zero. Here $\bar{S} = S \otimes_{O_{\breve E}} \bar k$. Our main conjecture can now be stated as follows. 
\begin{conjecture}\footnote{See the footnote \ref{footnoteScholze}}\label{conjDr}
 The  functor $ \CM_r$ is represented by $\hat{\Omega}^n_F \hat{\otimes}_{O_F} O_{\breve E}$.
\end{conjecture}
Our main results towards this conjecture are the following. 
\begin{theorem}\label{mainUR}
The conjecture is true if $F/\BQ_p$ is unramified. 
\end{theorem}
If $F/\BQ_p$ is ramified, we can still prove the following properties of $ \CM_r$ which are analogous to the properties of $\hat{\Omega}^n_F \hat{\otimes}_{O_F} O_{\breve E}$.
\begin{theorem}\label{mainDR}
The formal scheme $ \CM_r$ is flat over $\Spf\, O_{\breve E}$, and is $\pi$-adic. All its completed local rings are normal. Furthermore, 
\begin{altenumerate} 
\item there is an isomorphism between the special fibers 
$$ \CM_r\times_{\Spf O_{\breve E}}\Spec\, \bar k
\simeq \hat{\Omega}^n_F \times_{\Spf O_F}  \Spec\, \bar k.$$
\item there is an isomorphism between the generic fibers 
$$ \CM_r^{\rm rig}\simeq (\hat{\Omega}^n_F \times_{\Spf O_F} \Spf O_{\breve E})^{\rm rig} .$$
\end{altenumerate}
\end{theorem}
Here the proof of point (ii) is due to P.~Scholze. 
\medskip

We also prove the following  variant of this theorem in the Lubin-Tate
context. Let $F$ and $\varphi_0$ be as before, fix an integer $n\geq 2$
and let $r$ and $E=E_r$ have the same meaning as before. We now
consider $p$-divisible formal groups $X$ of height $n d$ over
$O_E$-schemes $S$ with an action $\iota : O_F \lra \End (X)$. We
impose the following Kottwitz condition, 
\begin{equation}\label{kottwitzLt}
 {\rm char} \, (\iota (x) \mid \Lie X) = \prod\nolimits_{\varphi} (T -
 \varphi (x))^{r_\varphi} \, , \,\, x \in O_F \, . 
\end{equation}
In addition, we impose  Eisenstein conditions (where the second condition in \eqref{furtherDr2} is changed into  
$$
\bigwedge^{2} \big( Q_{A_{\psi_0}}(\iota(\pi) \vert\Lie_{\psi_0}
    X) \big) =  0 ,
$$ see
(\ref{EisensteinFe})). 

We fix a pair $(\BX, \iota_{\BX})$ over $\bar k$ as above. It is easy
to see that 
$\BX$ is unique up to $O_F$-linear isogeny. We may therefore define a functor $\CM^F_r$ on ${\rm Nilp}_{O_{\breve{{E}}}}$ analogous to the functor $\CM_r$ above. The formal scheme representing this functor
will be denoted by the same symbol. If $r=r^\circ$ is the Drinfeld
function, then $\CM^F_{r^\circ}$ can be identified with the Lubin-Tate
deformation space (this follows from Proposition \ref{comprDr}).  Our
main result in this context is that this continues to hold for
arbitrary $r$.  
\begin{theorem}\label{mainLT}
 The formal scheme $\CM^F_r$ is isomorphic to $\Spf O_{\breve{E}}
 [[t_1, \cdots , t_{n-1}]]$. 
\end{theorem}

\section{The Kottwitz and Eisenstein conditions} \label{s:kotteis}

In this section, we analyze the conditions that can be put on a
locally free module with $O_D$-action. We continue with the same
notation as before.  
In addition, let $\tilde{E} \subset
\bar{\mathbb{Q}}_p$ be a normal extension of $\mathbb{Q}_p$ which
contains the images of all $\mathbb{Q}_p$-algebra homomorphisms
$\tilde{F} \rightarrow \bar{\mathbb{Q}}_p$. We have $E \subset
\tilde{E}$.  

We denote by
$\Nrd_{\varphi}: D \otimes_{F, \varphi} \tilde{E} \rightarrow
\tilde{E}$ the reduced norm. Using it, we define the polynomial function 
\begin{equation}\label{Nrdr1e}
\Nrd_{r}: D \otimes_{\mathbb{Q}_p} \tilde{E} \cong \prod_{\varphi} D
\otimes_{F, \varphi} \tilde{E} \; \overset{\prod_{\varphi}
  \Nrd_{\varphi}^{r_{\varphi}}}{\longrightarrow} \; \tilde{E}.  
\end{equation} 
If $\mathcal{M}$ is a quasicoherent sheaf on a scheme $S$ we denote by 
$\mathbb{V}_S(\CM )(T) = \Gamma(T,\mathcal{M}_T)$ the
corresponding flat sheaf on the category of $S$-schemes $T$. This
sheaf is representable by a scheme over $S$ if $\mathcal{M}$ is a
finite locally free $\mathcal{O}_S$-module.  

We write simply $\mathbb{V}(D)$ for the affine space over $\mathbb{Q}_p$
associated to $D$. Then we may regard $\Nrd_{r}$ as a polynomial
function (= morphism of schemes). It is defined over $E$,
\begin{equation}\label{Nrdr2e}
\Nrd_{r}: \mathbb{V}(D)_{E} \rightarrow \mathbb{A}^{1}_E.   
\end{equation}
Clearly this function is homogeneous of degree $\sum_{\varphi}
nr_{\varphi}$.

Let $\SuetO_{D}$ be the ring of integers of $D$. Let $\tilde{F}
\subset D$ be an unramified extension of $F$ of degree $n$. We denote
by $\tau \in \Gal(\tilde{F}/F)$ the Frobenius automorphism. Then we
may write
\begin{equation}\label{OD1e}
   \SuetO_D = \SuetO_{\tilde{F}}[\Pi], \quad \Pi^n = \pi, \quad \Pi a
   = \tau(a) \Pi, \; \text{for} \; a \in \SuetO_{\tilde{F}}. 
\end{equation}
Here $\Pi$ is a prime element of $\SuetO_D$ and $\pi$ is a prime
element of $\SuetO_{F}$. 

 We denote by $\mathbb{V}(\SuetO_{D})$ the corresponding affine space
 over $\mathbb{Z}_p$. We will now define an integral version of
 (\ref{Nrdr2e}).     

% Let $R$ be a $\SuetO_E$-algebra. Let $L$ be a finitely generated
% projective $R$-module wie an action of $\SuetO_D$:
% \begin{displaymath}
% \SuetO_D \rightarrow \End_{R} L. 
% \end{displaymath}

We begin with a general remark. Let $S$ be an $\SuetO_E$-scheme and let
$\mathcal{L}$ be a finite 
locally free $\mathcal{O}_S$-module with an action of $\SuetO_D$,
i.e. a morphism of $\mathbb{Z}_p$-algebras
\begin{equation}\label{ODmod1e}
\iota:  \SuetO_D \rightarrow \End_{\mathcal{O}_S}(\mathcal{L}).  
\end{equation}
We will call $(\mathcal{L},\iota)$ an {\it $\SuetO_D$-module over $S$}. 

If $T$ is an $S$-scheme and $\alpha \in \Gamma(T, \SuetO_{D}
\otimes_{\mathbb{Z}_p} \mathcal{O}_T)$,  we can take the determinant
$\det (\alpha | \mathcal{L}_T )$. This defines a morphism
\begin{displaymath}
\det\textstyle{_{\mathcal{L}}}: \mathbb{V}(\SuetO_D)_{S} \rightarrow
\mathbb{A}^{1}_S.  
\end{displaymath}
Let $\varphi \in \Phi$. We define an embedding
\begin{displaymath}
\SuetO_D \otimes_{\suetO_F, \varphi} \SuetO_{\tilde{E}} \rightarrow
M(n\times n, \SuetO_{\tilde{E}}).   
\end{displaymath}
For this we choose an embedding $\tilde{\varphi}: \tilde{F}
\rightarrow \bar{\mathbb{Q}}_p$ which extends $\varphi$ and
define for $x \in \tilde{F}$, 
 
\begin{equation}\label{Pphi1e}
x \mapsto 
\left(
\begin{array}{cccc}
\tilde{\varphi}(x) & 0 & \ldots & 0\\
0 & \tilde{\varphi}(\tau(x)) & \ldots & 0\\
\multicolumn{4}{c}{\ldots}\\
0 & 0 & \ldots & \tilde{\varphi}(\tau^{n-1}(x))
\end{array}
\right),
\; \Pi \mapsto
\left(
\begin{array}{ccccc}
0 & 1 & 0 & \ldots & 0\\
0 & 0 & 1 & \ldots & 0\\
\multicolumn{5}{c}{\ldots}\\
\varphi(\pi) & 0 & 0 & \ldots & 0
\end{array}
\right).
\end{equation} 

\medskip

Consider the standard $M(n\times n, \SuetO_{\tilde{E}})$-module
$\SuetO_{\tilde{E}}^{n}$. Via restriction of scalars 
\begin{displaymath}
\SuetO_D \otimes_{\mathbb{Z}_p} \SuetO_{\tilde{E}} \rightarrow
\SuetO_D \otimes_{\suetO_F, \varphi} \SuetO_{\tilde{E}} \rightarrow 
M(n\times n, \SuetO_{\tilde{E}}).   
\end{displaymath}
we obtain a $\SuetO_D \otimes_{\mathbb{Z}_p} \SuetO_{\tilde{E}}$-module
$P_{\varphi}$. We define the 
$\SuetO_D \otimes_{\mathbb{Z}_p} \SuetO_{\tilde{E}}$-module 
\begin{displaymath}
P_r = \bigoplus P_{\varphi}^{r_{\varphi}}. 
\end{displaymath}
This module defines a polynomial function
$\Nrd_{r}: \mathbb{V}(\SuetO_D)_{\suetO_{\tilde{E}}} \rightarrow
\mathbb{A}^{1}_{\suetO_{\tilde{E}}} $,  
\begin{displaymath}
\Nrd_{r}(\xi) = \det(\xi |P_r \otimes_{\suetO_{\tilde{E}}} R),
\quad \xi \in  \SuetO_D \otimes_{\mathbb{Z}_p} R,
\end{displaymath}
where $R$ is an arbitrary $\SuetO_{\tilde{E}}$-algebra. Similarly we
have polynomial functions $\Nrd_{\varphi}:
\mathbb{V}(\SuetO_D)_{\suetO_{\tilde{E}}} \rightarrow 
\mathbb{A}^{1}_{\suetO_{\tilde{E}}}$ given by $P_{\varphi}$.  
The polynomial function $\Nrd_{r}$ is invariant under $\Gal(\tilde{E}/E)$ and
therefore defines a polynomial function
\begin{equation}\label{normmorph}
\Nrd_r: \mathbb{V}(\SuetO_D)_{\suetO_{E}} \rightarrow
\mathbb{A}^{1}_{\suetO_{E}}.
\end{equation}
It follows from (\ref{Pphi1e}) that $\Pi$ acts on $P_{\varphi}
\otimes_{O_{\tilde{E}}} \kappa_{\tilde{E}}$ as zero. Therefore the
latter is a $(O_D/\Pi O_D) \otimes_{O_{\tilde{F}}}
\kappa_{\tilde{E}}$-module. We see that the base change
$\Nrd_{\varphi}$ factors through a polynomial function
\begin{displaymath}
(\Nrd_{\varphi})_{\kappa_{\tilde{E}}}:
\mathbb{V}(\SuetO_D/\Pi\SuetO_D)_{\kappa_{\tilde{E}}} \rightarrow
\mathbb{A}^{1}_{\kappa_{\tilde{E}}}.  
\end{displaymath}
We obtain that $(\Nrd_{r})_{\kappa_{\tilde{E}}}$ factors too,
\begin{equation*}
(\Nrd_r)\kappa_{\tilde{E}}:
\mathbb{V}(\SuetO_D/\Pi\SuetO_D) \times_{\Spec \mathbb{F}_p} \Spec
\kappa_{\tilde{E}} \rightarrow \mathbb{A}^{1}_{\kappa_{\tilde{E}}}.  
\end{equation*}
The affine algebra on the left hand side is an integral domain. 
Therefore $(\Nrd_{\varphi})_{\kappa_{\tilde{E}}}$ is a non zero-divisor
as an element of this algebra. This remains true after base change to
any $\kappa_{\tilde{E}}$-algebra $R$.

% The homomorphism $D \otimes \tilde{E} \rightarrow D \otimes_{F,
%   \varphi} \tilde{E}$  
% induces a morphism of schemes $\mathbb{V}(\SuetO_D)_{\suetO_{\tilde{E}}}
% \rightarrow \mathbb{V}_{\suetO_F}(\SuetO_D)_{\varphi,\suetO_{\tilde{E}}}$.  
% The reduced norm $\Nrd: D \rightarrow F$ induces a morphism 
% \begin{displaymath}
% \Nrd: \mathbb{V}_{\suetO_F}(\SuetO_D) \rightarrow
% \mathbb{A}^{1}_{\suetO_F}.  
% \end{displaymath}
% We denote by $\Nrd_{\varphi}$ the morphism obtained by base change via
% $\SuetO_F \overset{\varphi}{\longrightarrow} \SuetO_{\bar{E}}$. We obtain
% an integral version of  
% (\ref{Nrdr1e}):
% \begin{equation}\label{Nrdr2e}
% \Nrd_r: \mathbb{V}(\SuetO_D)_{\suetO_{\tilde{E}}} \longrightarrow
% \prod_{\varphi}
% \mathbb{V}_{\suetO_F}(\SuetO_D)_{\varphi,\suetO_{\tilde{E}}}
% \longrightarrow \mathbb{A}^{1}_{\suetO_{\tilde{E}}},  
% \end{equation}
% where the last map is the product of the
% $\Nrd_{\varphi}^{r_{\varphi}}$. 

% Since (\ref{Nrdr1e}) is defined over
% $E$ we obtain a morphism of $\SuetO_E$-schemes:
% \begin{equation}\label{Nrdr3e}
% \Nrd_r: \mathbb{V}(\SuetO_D)_{\suetO_{E}} \longrightarrow
% \mathbb{A}^{1}_{\suetO_{E}},
% \end{equation}

\begin{definition} Let $S$ be an $O_E$-scheme and let $(\CL, \iota)$ be an $O_D$-module over $S$. 
We say that {\it $(\mathcal{L}, \iota)$ satisfies the Kottwitz condition
$(\mathbf{K}_r)$ with respect to $r$}, if
\begin{equation}\label{KCond1e}
\det\textstyle{_{\mathcal{L}}} = (\Nrd_{r})_S,
\end{equation}
where the right hand side is the base change with respect to
$S \rightarrow \Spec \SuetO_{E}$.  
\end{definition}
\begin{remark}\label{kottchar}
By \cite{H}, Prop. 2.1.3, the condition is equivalent to the identity of polynomials in $\CO_S[T]$ (comp. \eqref{charpolDr})
\begin{equation}
 {\rm char} (\iota (x) \vert \CL) = \prod\nolimits_{\varphi} \,
 \,\varphi\big( {\rm chard} \, \, (x) (T)\big)^{r_\varphi} ,  
\quad \forall x \in O_D\, ;
\end{equation}
(this uses Amitsur's formula, comp. \cite{Ch}, Lemma~1.12). 
\end{remark}
It is clear that the Kottwitz condition is a closed condition. If the Kottwitz condition is fulfilled, we have
\begin{displaymath}
   \rank_S \mathcal{L} = n \sum_{\varphi} r_{\varphi} ,
\end{displaymath}
because $\Nrd_{r}$ is homogeneous of this degree.

Recall  the maximal unramified subfield $F^{t} \subset F$ and its
residue field $\kappa$. We denote  by $\tilde{\kappa}$ the residue
field of $\tilde{F}$. The maximal 
unramified subfield of $\tilde{F}$ is denoted by $\tilde{F}^{t}$.  
In addition to  $\Psi =
\Hom_{\mathbb{Q}_p}(F^{t},\bar{\mathbb{Q}}_{p})$,  we introduce
$\tilde{\Psi} = \Hom_{\mathbb{Q}_p}(\tilde{F}^{t},
\bar{\mathbb{Q}}_{p})$. 

We now introduce another condition which will turn out to be  weaker
than the Kottwitz condition. 
Let $R$ be an $\SuetO_{\tilde{E}}$-algebra. Let $(L,\iota)$ be an
$\SuetO_D$-module over $R$. Then we have the decompositions
\begin{equation}\label{decompunr}
L = \bigoplus_{\psi \in \Psi} L_{\psi}, \quad L = 
\bigoplus_{\tilde{\psi} \in \tilde{\Psi}} L_{\tilde{\psi}}. 
\end{equation}

For example, the second of these decompositions is induced by
\begin{displaymath}
\SuetO_{\tilde{F}^{t}} \otimes_{\mathbb{Z}_p} \SuetO_{\tilde{E}} =
\prod\nolimits_{\tilde{\psi} \in \tilde{\Psi}} \SuetO_{\tilde{E}}.
\end{displaymath}

For $\psi \in \Psi$ let $\Phi_{\psi}$ the set of all embeddings $\varphi:
F \rightarrow \bar{\mathbb{Q}}_p$ whose restriction to $F^{t}$ is
$\psi$. We define  
\begin{equation}\label{KC4e}
r_{\psi} = \sum_{\varphi \in \Phi_{\psi}} r_{\varphi}
\end{equation}

\begin{definition}\label{rkcond1d} Let $R$ be an $O_{\tilde E}$-algebra. 
We say that an $\SuetO_D$-module $(L, \iota)$ over $R$ {\it satisfies the
rank condition $(\mathbf{R}_r)$ with respect to $r$} if, for all $\tilde{\psi} \in
\tilde{\Psi}$, 
\begin{equation}\label{KC5e}
\rank_R L_{\tilde{\psi}} = r_{\psi},
\end{equation}
where $\psi$ denotes the restriction of $\tilde{\psi}$ to $F^{t}$.
We will write $r_{\tilde{\psi}} := r_{\psi}$. 
\end{definition}
The rank condition is obviously an open condition (the rank goes up under specialization). It is also a closed condition since $\sum\nolimits_{\tilde \psi} \rank_R L_{\tilde{\psi}}=\rank L$ is constant on the base. To check the rank condition it is enough to check it for the geometric
points of $\Spec R$. If $R$ is an arbitrary $\SuetO_{E}$-algebra, we say that
the rank condition is fulfilled, if it is fulfilled with respect to
any base change $R \rightarrow  \tilde{R}$ and an extension of the
$\SuetO_E$-algebra structure to an $\SuetO_{\tilde{E}}$-algebra
structure on $\tilde{R}$. The rank condition is then independent of
the last choice.

\begin{remark}  The rank condition is  independent of the choice of
the chosen isomorphism (\ref{OD1e}). Indeed let $U \subset D$
be an unramified extension of degree $n$ of $F$. Then we could
reformulate the rank condition using the action of $\SuetO_{U^t}
\otimes_{\mathbb{Z}_p} \SuetO_{\tilde{E}}$. However, this yields the same
condition. Indeed, since $U$ and $\tilde{F}$ are isomorphic field
extensions of $F$,  there is by Skolem-Noether an element $u \in D$,
such that $u \tilde{F} u^{-1} = U$. Replacing $u$ by $u\Pi^{m}$ for a
suitable integer $m$, we may assume that $u \in \SuetO_D^\times$. We denote
by $L_{[u]}$ the $R$-module $L$ with the new $\SuetO_D$-action 
\begin{displaymath}
   a \cdot_{new} \ell = u a u^{-1} \ell. 
\end{displaymath}
But the  decomposition (\ref{decompunr}) for $L_{[u]}$, 
\begin{displaymath}
   L_{[u]} =  \bigoplus_{\tilde{\psi} \in \tilde{\Psi}} (L_{[u]})_{\tilde{\psi}} \, ,
\end{displaymath}
is exactly the decomposition coming from the $\SuetO_{U^t}$-action on
$L$. Since the multiplication by $u$ $ L \rightarrow L_{[u]}$ is an
isomorphism of $\SuetO_{D}$-modules we obtain the independence. 
\end{remark}

We denote by $\kappa_{\tilde{E}}$ the residue class field of
$\SuetO_{\tilde{E}}$.   
We consider a polynomial function over a {\it reduced}
$\kappa_{\tilde{E}}$-algebra $R$,  
\begin{equation}\label{KC2e}
\chi: \mathbb{V}(\SuetO_D)_{R} \rightarrow
\mathbb{A}^{1}_{R},  
\end{equation}
which is multiplicative with respect to the ring structures of these
schemes. We are given for each $R$-algebra $A$ a multiplicative map
\begin{equation}\label{KC1e}
\chi_A: \SuetO_D/p\SuetO_D \otimes_{\mathbb{F}_p} A \rightarrow A.  
\end{equation}
Let $a,b \in 
\SuetO_D/p\SuetO_D \otimes_{\mathbb{F}_p} A$. We claim that 
\begin{equation}\label{poly1e}
\chi ((a + (\Pi \otimes 1)b)) = \chi (a). 
\end{equation} 
We regard this as an identity of polynomial functions on
$\mathbb{V}(\SuetO_D)_R \times_{R} \mathbb{V}(\SuetO_D)_R$. We
consider the units of $\SuetO_D$ as a subscheme $G \subset 
\mathbb{V}(\SuetO_D)_R$. This is dense in each fiber over
$R$. Therefore it suffices to show (\ref{poly1e}) in the case where
$a$ is a unit. By  multiplicativity it suffices to show that 
\begin{equation}\label{poly2e}
\chi ((1 + (\Pi \otimes 1)b)) = \chi (1). 
   \end{equation}
We may restrict our attention to the universal case where $A = R[Y_1,
\ldots, Y_t]$, which is also reduced.  

We have 
\begin{displaymath}
(1 + (\Pi \otimes 1)b)^{p^{s}} = 1^{p^{s}}
\end{displaymath}
for some $p$-power $p^{s}$.  Since $A$ is reduced and $\chi$ is
multiplicative,  we deduce (\ref{poly2e}).

Therefore (\ref{KC1e}) is equivalent to a functorial map 
\begin{displaymath}
\SuetO_D/\Pi \SuetO_D \otimes_{\mathbb{F}_p} A \rightarrow A.  
\end{displaymath}
We have $\tilde{\kappa} = \SuetO_D/\Pi \SuetO_D$. Therefore for each
$\tilde{\psi} \in \tilde{\Psi} = \Hom_{\mathbb{F}_p}(\tilde{\kappa},
\kappa_{\tilde{E}}) = \Hom_{\mathbb{Q}_p}(\tilde{F}^{t},
\bar{\mathbb{Q}}_p)$,  we obtain a polynomial function
\begin{equation}\label{KCchi1e}
\chi_{\tilde{\psi}}: \SuetO_D \rightarrow \tilde{\kappa}
\overset{\tilde{\psi}}{\rightarrow} \kappa_{\tilde{E}} \rightarrow A.  
\end{equation}
If $\Spec R$ is connected and reduced, we deduce that the only
multiplicative polynomial functions (\ref{KC2e}) have the form
\begin{equation}\label{KCchi2e}
\prod_{\tilde{\psi}} \chi_{\tilde{\psi}}^{e_{\tilde{\psi}}},
\end{equation}
for suitable exponents $e_{\tilde{\psi}}$. These functions are also defined if $R$ is an arbitrary
$\kappa_{\tilde{E}}$-algebra. 
\begin{lemma}\label{KC0l}
Let $R$ be a reduced $\kappa_{E}$-algebra. Let $(L, \iota)$ be
an $\SuetO_D$-module over $R$. In particular $L$ is a 
finitely generated locally free $R$-module. 

Then the Kottwitz condition $(\mathbf{K}_r)$ for $(\mathcal{L},
\iota)$ is equivalent to the rank condition $(\mathbf{R}_r)$.
For an arbitrary $\SuetO_E$-algebra $R$ the condition $(\mathbf{K}_r)$
implies the condition $(\mathbf{R}_r)$.
\end{lemma}
\begin{proof} 
To prove this we make a base change $R \mapsto R \otimes_{O_E}
O_{\tilde{E}}$ which is reduced if $R$ is a reduced
$\kappa_E$-algebra. Therefore we may assume that $R$ is a
$O_{\tilde{E}}$-algebra. 

The first assertion follows because a polynomial function is uniquely
determined by the numbers $e_{\tilde{\psi}}$ in (\ref{KCchi2e}). 

The last assertion depends only on the geometric fibers. In this case
we have either a $\kappa_{\tilde{E}}$-algebra or an
$\SuetO_{E}$-algebra of characteristic $0$. We have already seen the
first case. The characteristic $0$ case is clear. 

\end{proof} 

\begin{proposition}\label{KC1p}
Let us assume that $r=r^\circ$, i.e., $r_{\varphi} = 0$ for $\varphi
\neq \varphi_0$.  
Let $R$ be a $\kappa_{\tilde{E}}$-algebra. Let $(L,\iota)$ be
a $\SuetO_D$-module over $R$. 
Then $(L, \iota)$ satisfies the condition $(\mathbf{K}_r)$ iff $\pi$
annihilates $L$ and the rank condition $(\mathbf{R}_r)$ is
fulfilled. 

The rank condition says in this case that for all
$\tilde{\psi}: \tilde{\kappa} \rightarrow  \kappa_{\tilde{E}}$
\begin{displaymath}
\rank L_{\tilde{\psi}} = 1, \; \; \text{if} \; \;
\tilde{\psi}_{|F^{t}} =\psi_0= {\varphi_0} _{|F^{t}}, \quad \rank
L_{\tilde{\psi}} = 0, \; \text{else}. 
\end{displaymath}
\end{proposition}
\begin{proof} By Lemma \ref{KC0l} we already know that the Kottwitz condition
implies the rank conditions. If $a \in (\SuetO_D/p\SuetO_D) \otimes
R$ we have $\Nrd_r(a\pi) = 0$. Therefore the Kottwitz condition
implies $\det(\pi | L_{\tilde{\psi}}) = 0$. Since $L_{\tilde{\psi}}$
has rank $1$, this shows that $\pi$ annihilates $L$.  

Assume conversely that the ranks are as indicated and that $\pi$
annihilates $L$. Then we deduce the result from the following
Lemma. 
\end{proof}

\begin{lemma}\label{KC1l} 
Let $R$ be a ring, such that  $pR = 0$. Let $n$ be a
natural number. Let $L_1, \ldots, L_n$ be locally free $R$-modules of
rank $1$. Assume we are given a chain of homomorphisms $\Pi_i:  L_i
\rightarrow L_{i+1}$ and $\Pi_n: L_n \rightarrow L_1$ such that
\begin{displaymath}
   \Pi_n\circ \Pi_{n-1} \circ \ldots \circ \Pi_{1} = 0. 
\end{displaymath} 
We set $L = L_1  \bigoplus \ldots  \bigoplus L_n$, and  $\Pi = \Pi_1  \bigoplus
\ldots  \bigoplus \Pi_n$. This is an endomorphism of $L$ such that $\Pi^n
= 0$. 

Let $v \in M(n\times n, R)$ be a diagonal matrix. It induces an
endomorphism $v: L \rightarrow L$. We consider an endomorphism of $L$
of the form
\begin{displaymath}
   v_0 + v_1\Pi + v_2\Pi^2 + \ldots + v_{n-1}\Pi^{n-1},
\end{displaymath} 
where the $v_i$ are diagonal matrices.

Then
\begin{equation}\label{Kdet1e}
   \textstyle\det_F (v_0 + v_1\Pi + v_2\Pi^2 + \ldots + v_{n-1}\Pi^{n-1}) =
   \det_F v_0. 
\end{equation} 
\end{lemma}
\begin{proof} We may assume $L_i = R$. Then $\Pi_i$ is the
multiplication by some element $y_i \in R$. Our assumption says
\begin{displaymath}
   y_n y_{n-1} \cdot \ldots \cdot y_1 = 0. 
\end{displaymath}
One deduces that $\Pi^n = 0$. We note that $\Pi v = v' \Pi$, where
$v'$ is another diagonal matrix, which is obtained from $v$ by
cyclically permuting the diagonal entries. 

We may reduce to the universal
case where $R$ is the quotient of a polynomial ring $R =
\mathbb{F}_p[x_{ij}, y_k]/y_1\cdot \ldots \cdot y_n$, where $i \in
[0,n-1]$, $j \in [1,n]$, $k \in [1,n]$. For fixed $i$, the $x_{ij}$ are
the diagonal entries of $v_i$. In this case the ring $R$ is
reduced. 

Consider first the case where $v_0 = E$ is the unit matrix. We take
the universal case where the ring $R$ is reduced (as above but no
indeterminates $x_{0j}$). 
Then we have
two commuting operators $v_0$ and $\rho = v_1\Pi + v_2\Pi^2 + \ldots +
v_{n-1}\Pi^{n-1}$. Therefore
\begin{displaymath}
   (E + \rho)^{p^s} = E^{p^s} + \rho^{p^s}
\end{displaymath}
But $\rho^{p^s} = 0$ for $p^s > n$. This implies $\det (E + \rho)^{p^s}
= 1$. Because $R$ is reduced in the universal case,  this implies
(\ref{Kdet1e}) in the case where $v_0 = E$.  Clearly this shows also
the case where $v_0$ is an invertible diagonal matrix.

In the general case we consider the universal $R$ as above and its 
localization,
\begin{displaymath}
   R \subset R[x_{01}^{-1}, x_{02}^{-1}, \ldots, x_{0n}^{-1}] .
\end{displaymath}     
Over the bigger ring $v_0$ becomes invertible and therefore the
relation (\ref{Kdet1e}) holds. Hence it holds also over the subring
$R$. 
\end{proof}

Let us recall the Eisenstein conditions. We recall  the following notation: For an embedding $\psi: F^{t}
\rightarrow \bar{\mathbb{Q}}_p$ we set
\begin{equation*}
\begin{aligned}
   A_\psi &= \{\varphi: F \rightarrow \bar{\mathbb{Q}}_p \;|\;
   \varphi_{|F^{t}} = \psi, \; \text{and} \; r_{\varphi} = n  \}\\
B_{\psi} &= \{\varphi: F \rightarrow \bar{\mathbb{Q}}_p \;|\;
   \varphi_{|F^{t}} = \psi, \; \text{and} \; r_{\varphi} = 0  \}.
\end{aligned}
\end{equation*} 
We set $\psi_0 = {\varphi_0} _{|F^{t}}$. This gives a partition of the
set $\Phi_{\psi}$:
\begin{displaymath}
\Phi_{\psi_0} = A_{\psi_0} \cup B_{\psi_0} \cup \{ \varphi_0 \}, \quad
\Phi_{\psi} = A_{\psi} \cup B_{\psi} \, \text{for} \, \psi \neq \psi_0.
\end{displaymath}

Let $a_{\psi} = |A_{\psi}|$,
$b_{\psi} = |B_{\psi}|$. Then we have
\begin{displaymath}
   a_{\psi} + b_{\psi} + \epsilon_{\psi} = e = [F:F^{t}],
\end{displaymath}
where $\epsilon_{\psi} = 0$ if $\psi \neq \psi_0$, and
$\epsilon_{\psi_0} = 1$.   

We find (compare (\ref{KC4e})) that 
\begin{equation}\label{KC5e2}
r_{\psi} = 
\begin{cases}
\begin{array}{ll}
na_{\psi}+1 &  \text{if} \; \psi = \psi_0\\
na_{\psi}  & \text{if} \; \psi \neq \psi_0.
\end{array}
\end{cases}
\end{equation}

Let $Q(T) \in \SuetO_{F^t}[T]$ be the Eisenstein polynomial of a fixed 
prime element $\pi \in F$.   We set
$Q_{\psi}(T) = \psi(Q(T)) \in \SuetO_E[T]$. We set
\begin{displaymath}
   Q_{A_{\psi}}(T) = \prod_{\varphi \in A_{\psi}} (T - \varphi(\pi)), \quad
   Q_{B_{\psi}}(T) = \prod_{\varphi \in B_{\psi}} (T - \varphi(\pi)).
\end{displaymath}  
These are polynomials in $\SuetO_E[T]$.  Moreover we set $Q_0(T) =
T - \varphi_0(\pi)$. Then we have the decompositions
\begin{equation}\label{Eisenstein1e}
\begin{aligned}
   Q_{\psi} (T) &= Q_{A_{\psi}}(T) \cdot Q_{B_{\psi}}(T), \quad
   \text{for} \; \psi \neq \psi_0\\
  Q_{\psi_0} (T) &= Q_0(T) \cdot Q_{A_{\psi_0}}(T) \cdot Q_{B_{\psi_0}}(T).
\end{aligned}
\end{equation}

Let $R$ be an $\SuetO_E$-algebra. We will introduce the Eisenstein
condition on a $\SuetO_D$-module $(L,\iota)$ over $R$. 
We have
decompositions 
\begin{displaymath}
\SuetO_{F^{t}} \otimes_{\mathbb{Z}_p} \SuetO_E \cong \prod_{\psi \in
  \Psi} \SuetO_E,   
\qquad   \quad 
\SuetO_{F^{t}} \otimes_{\mathbb{Z}_p} R
\cong \prod_{\psi \in \Psi} R. 
\end{displaymath}

This gives a decomposition
\begin{displaymath}
L =  \bigoplus_{\psi \in \Psi} L_{\psi}.
\end{displaymath}

\begin{definition}\label{Eisen1d}
We say that $(L,\iota)$ {\it satisfies the Eisenstein condition
$(\mathbf{E}_r)$ with respect to $r$} if
\begin{equation}\label{EisensteinCe}
\begin{array}{rrrr}
((Q_0 \cdot Q_{A_{\psi_0}})(\iota(\pi)) | L_{\psi_0}) & = & 0, &\\[2mm]
\bigwedge^{n+1}(Q_{A_{\psi_0}}(\iota(\pi) | L_{\psi_0})) & = & 0, &\\[2mm]
(Q_{A_{\psi}}(\iota(\pi)) | L_{\psi}) & = & 0, & \text{for} \;
\psi \neq \psi_0.
\end{array}
\end{equation}
This definition applies to any $\mathcal{O}_{E}$-scheme $S$ and any $O_D$-module $(\CL, \iota)$ over $S$. 
\end{definition}
\begin{remarks}\label{eisrem}
 (i) The condition $(\mathbf{E}_r)$ only depends on the restriction to $O_F$ of the action $\iota$ of $O_D$ on $\CL$.  The condition depends on the choice of the uniformizer $\pi$. 

\smallskip

\noindent (ii) Let $r=r^\circ$, i.e., $r_{\varphi} = 0$ for $\varphi \neq
\varphi_0$.  In this case $E=F$. We have two actions of $\SuetO_F$ on $\mathcal{L}$. The
first is given by $\iota$ and the second by $\SuetO_{F}
\overset{\varphi_0}{\longrightarrow} \SuetO_E \rightarrow
\mathcal{O}_S$. By the Eisenstein condition these actions
coincide. Indeed, $\CL=\CL_{\psi_0}$, and 
$a_{\psi_0} = 0$ for all $\psi \in \Psi$. Therefore
the Eisenstein condition (\ref{EisensteinCe}) implies that $\iota(\pi)$ acts
on $\mathcal{L}$ as $\varphi_0(\pi)$. The conditions (\ref{KC5e}) imply
moreover that $\iota(a)$ for $a \in \SuetO_{F^{t}}$ acts via $\varphi_0(a)$ 
on $\mathcal{L}$.  

\smallskip

\noindent (iii) Let $F/\BQ_p$ be unramified, i.e., $F=F^t$. Then the first and the last Eisenstein conditions are empty. The second Eisenstein condition just says 
that $\rank \CL_{\varphi_0}\leq n$.  

\smallskip

\noindent (iv) Let us assume that $S$ is a $\kappa_E$-scheme. The images of the
three polynomials $Q_{A_{\psi}}(T)$, $Q_{B_{\psi}}(T)$ and $Q_0(T)$
in $\kappa_E[T]$ are respectively
\begin{displaymath}
   T^{a_{\psi}}, \quad T^{b_{\psi}}, \quad T.
\end{displaymath}
Therefore the Eisenstein conditions are in this case:
\begin{equation}\label{Eisenstein2e}
   \begin{aligned}
\iota(\pi)^{a_{\psi_0}+1} &= 0  \text{ on }  \CL_{\psi_0},\\
\wedge^{n+1}( \iota(\pi)^{a_{\psi_0}}) &= 0  \text{ on } 
\wedge^{n+1} \CL_{\psi_0},\\ 
\iota(\pi)^{a_{\psi}} &= 0   \text{ on }  \CL_{\psi}  \text{ for } 
\psi \neq \psi_0.
\end{aligned}
\end{equation}
\end{remarks}

\begin{definition}\label{Dr1d}
Let $(\mathcal{L}, \iota)$ be a $\SuetO_D$-module over an
$\SuetO_E$-scheme $S$. We say that {\it $(\mathcal{L}, \iota)$ satisfies
the Drinfeld condition $(\mathbf{D}_r)$ with respect to $r$} if it
satisfies the Eisenstein condition $(\mathbf{E}_r)$ and the rank
condition $(\mathbf{R}_r)$. 
\end{definition}
The rank condition makes sense by the remark after Definition
\ref{rkcond1d}.

\begin{lemma}\label{KrDrDrin}
Assume that $r=r^\circ$, i.e., $r_{\varphi} = 0$ for $\varphi \neq \varphi_0$. Let $S$ be
a $\kappa_E$-algebra. 

Then the condition $(\mathbf{D}_r)$ implies $(\mathbf{K}_r)$. 
\end{lemma}
This is a reformulation of Proposition \ref{KC1p}.

\section{Formal $\SuetO_D$-modules}\label{s:formalmod}

\begin{definition}\label{ODmod1d} Let $S$ be an  $\SuetO_E$-scheme
  such that $p$ is nilpotent on $S$. 
A {\it $p$-divisible $\SuetO_D$-module over $S$} 
is a $p$-divisible group $X$ of height $[D:\mathbb{Q}_p] = n^2d$ with
an action
\begin{equation}\label{OD2e}
   \iota: \SuetO_D \rightarrow \End X.
\end{equation}
In this case $\Lie X$ is an $\SuetO_D$-module in the sense of
(\ref{ODmod1e}).  
We say that $X$ {\it satisfies $(\mathbf{K}_r)$ (\ref{KCond1e}), resp. $(\mathbf{R}_r)$ (\ref{KC5e}), resp. $(\mathbf{E}_r)$ (\ref{EisensteinCe}),
resp. $(\mathbf{D}_r)$} (Definition \ref{Dr1d}) if the
$\mathcal{O}_S$-module $\Lie X$ with the  $\SuetO_D$-action $d\iota$
does. 

A $p$-divisible $\SuetO_D$-module $X$ which satisfies $(\mathbf{D}_r)$
  is  called an {\it $r$-special formal
  $\SuetO_D$-module}. The name is justified because we will prove that
$X$ has no \'etale part. 
A formal $\SuetO_D$-module
which satisfies $(\mathbf{D}_{r^{\circ}})$ is also called a {\it
  special formal $\SuetO_D$-module}.  
\end{definition}
\begin{remarks}
(i) The definition of a special formal $\SuetO_D$-module above coincides with Drinfeld's definition in \cite{D}. Indeed, in this case $r=r^\circ$, and it follows from Remarks \ref{eisrem}, (ii) that $X$ is a strict formal $O_F$-module (the induced action of $O_F$ on $\Lie X$ is via the structure morphism $O_F=O_E\to\CO_S$). 

\smallskip

\noindent (ii) Let $A$ be a $p$-adic $\SuetO_E$-algebra. Let $X$ be a $p$-divisible
group over $A$ with an action (\ref{OD2e}). Then we define $\Lie X =
{\varprojlim} \Lie X \otimes A/p^{n}$. It is still a
$\SuetO_D$-module and the definition above makes sense.  
\end{remarks}
Assume that $p$ is locally nilpotent on $S$. 
We denote by $\mathbb{D}(X)$ the covariant crystal associated to
$X$. The evaluation $\mathbb{D}(X)_S$ at $S$ coincides with the
Lie algebra of the universal extension of $X$. 

Let $a \in \SuetO_D$. We will often write  simply $a$ when we mean
the action $\iota(a)$ on $\mathbb{D}(X)_S$, or the derived action
$d\iota(a)$ on $\Lie X$.

\begin{proposition}\label{DOD1p}
Assume that $p$ is locally nilpotent on $S$. Let $X$ be a 
$r$-special formal $\SuetO_D$-module of height $n^{2}d$.

Then $\mathbb{D}(X)_S$ is locally on $S$ a free $\SuetO_D
\otimes_{\mathbb{Z}_p} \mathcal{O}_S$-module of rank $1$.  
Moreover $X$ is a formal Lie group. 
\end{proposition}

\begin{proof}
We begin with the case where $S = \Spec k$ where $k$ is a perfect
field with a $\kappa_{\tilde{E}}$-structure. Let $W(k)$ be the ring of
Witt vectors and $M =
\mathbb{D}(X)_{W(k)}$ be the covariant Dieudonn\'e module. 

We have a bijection of field embeddings
\begin{displaymath}
  \tilde{\Psi} = \Hom(\tilde{\kappa}, k) =\Hom(\tilde{F}^t,
  W(k)\otimes \mathbb{Q}).
\end{displaymath}
We regard $\tilde{\psi} \in \tilde{\Psi}$ as a homomorphism
$\tilde{\psi}: \SuetO_{\tilde{F}^{t}} \rightarrow W(k)$.  We set 
\begin{displaymath}
   M_{\tilde{\psi}} = \{m \in M \; | \; \iota(a)m = \tilde{\psi}(a)m, \;
   \text{for} \; a \in \SuetO_{\tilde{F}^t} \}\, .
\end{displaymath}
We have a direct decomposition
\begin{equation}\label{DC1e}
   M =  \bigoplus M_{\tilde{\psi}}, \quad \tilde{\psi} \in \tilde{\Psi}.
\end{equation}
We denote by $\sigma$ the Frobenius acting on $W(k)$. The Verschiebung
induces a map 
\begin{displaymath}
   V: M_{\sigma \tilde{\psi}} \rightarrow
M_{\tilde{\psi}}.
\end{displaymath}
Therefore the rank of the $W(k)$-module $M_{\tilde{\psi}}$ is
independent of $\tilde{\psi}$ and therefore equal to $ne$. For each
$\tilde{\psi}: \tilde{F}^t \rightarrow W(k) \otimes \mathbb{Q}$ we
have $\tilde{\psi} \tau = \sigma^{f} \tilde{\psi}$, where $\tau$ is
from (\ref{OD1e}). Since
$r_{\tilde{\psi}}$ depends only on the restriction of $\tilde{\psi}$
to $F^t$,  we find $r_{\tilde{\psi}} = r_{\tilde{\psi} \tau}$. Obviously
$\Pi$ induces a map 
\begin{displaymath}
   \Pi: M_{\tilde{\psi}\tau} \rightarrow M_{\tilde{\psi}}.
\end{displaymath}  
We consider the commutative diagram 
\begin{equation}\label{KC6e}
   \begin{CD}
M_{\sigma\tilde{\psi}\tau} @>{\Pi}>> M_{\sigma\tilde{\psi}}\\
@V{V}VV  @VV{V}V\\
M_{\tilde{\psi}\tau} @>{\Pi}>> M_{\tilde{\psi}}\, .
   \end{CD}
\end{equation}
By the rank condition, the cokernels of the vertical maps are
$k$-vector spaces of dimension $r_{\tilde{\psi}\tau} =
r_{\tilde{\psi}}$. Therefore the lengths of the cokernels of both
horizontal $\iota(\Pi)$ are also the same. On the other hand we have 
\begin{displaymath}
   \sum\nolimits_{\tilde{\psi}} \length M_{\tilde{\psi}}/
   \Pi M_{\tilde{\psi}\tau} = \length M/\Pi M = nf.
\end{displaymath}
The last equality holds since by assumption $\length M/pM = n^2d$. We
conclude that
\begin{displaymath}
   \length M_{\tilde{\psi}}/  \Pi M_{\tilde{\psi}\tau} = 1,
   \quad \text{for} \; \tilde{\psi } \in \tilde{\Psi}.
\end{displaymath}
The last equation tells us that $M/\Pi M$ is a free
$\SuetO_D/\Pi \SuetO_D \otimes_{\mathbb{F}_p} k$-module of rank
$1$. By  Nakayama's lemma  it follows that $M$ is a free
$\SuetO_D \otimes_{\mathbb{Z}_p} W(k)$-module. This completes the
proof of the first assertion in the  case $S = \Spec k$. 

To show that $X$ is a formal group we have to show that $V$ is
nilpotent on $M/\iota(\Pi)M$. We consider the map induced by $V$ on
the cokernels of the horizontal maps of the diagram (\ref{KC6e}). For
each $\tilde{\psi}$ this map is either a bijection or it is zero. It
suffices to see that this map is zero for some $\tilde{\psi}$. If not, 
we would have a bijection for each $\tilde{\psi}$. We denote the cokernels
of the vertical maps by $\Lie_{\tilde{\psi}} X$. Then 
\begin{displaymath}
   \Lie X =  \bigoplus_{\tilde{\psi}} \Lie_{\tilde{\psi}} X.
\end{displaymath}
We conclude that $\iota(\Pi):  \Lie_{\tilde{\psi} \tau} X \rightarrow
\Lie_{\tilde{\psi}} X$ is bijective for each $\tilde{\psi}$. This is a
contradiction since $\iota(\Pi)$ is nilpotent on $\Lie X$. Therefore
$X$ is a formal group. 

A base change argument shows that our result is true if $S$ is the
spectrum of a field. In the general case we consider a point $s \in
S$. We find $m \in \mathbb{D}(X)_S \otimes \kappa (s)$ which is a
basis of this $\SuetO_D \otimes \kappa(s)$-module. We may assume that $S =
\Spec R$ and that $m$ is the image of an element $n \in
\mathbb{D}(X)_R$. It follows from Nakayama's Lemma that the
homomorphism induced by $n$
\begin{displaymath}
   \SuetO_D \otimes R \rightarrow \mathbb{D}(X)_R
\end{displaymath} 
is an isomorphism in an open neighbourhood of $s$.
\end{proof}
We state the following consequence separately: 
\begin{corollary}
For each $\tilde{\psi} \in \tilde{\Psi}$ the cokernel of 
\begin{displaymath}
\Pi: \mathbb{D}(X)_{S, \tilde{\psi}\tau} \rightarrow
\mathbb{D}(X)_{S, \tilde{\psi}} 
\end{displaymath}
is a locally free $\mathcal{O}_S$-module of rank 1.\qed
\end{corollary}

\begin{proposition}\label{flatimplies}
Let $A$ be an $\SuetO_{E}$-algebra. Assume moreover that $A$ is a
$p$-adic integral domain, with fraction field  of characteristic
$0$. Let $X$ be an $r$-special formal $\SuetO_D$-module over
$A$.

Then  $\Lie X$ satisfies  $(\mathbf{K}_r)$.
\end{proposition}     
\begin{proof} Let $K$ denote the fraction field of $A$. By Proposition \ref{DOD1p},  $H = \mathbb{D}(X)_K$ is a free
$\SuetO_D \otimes K$-module of rank $1$. We consider the decomposition
\begin{displaymath}
H =  \bigoplus_{\psi \in \Psi} H_{\psi}.
\end{displaymath}
We can assume that $A$ is a $\SuetO_{\tilde{E}}$-algebra. 
Let $\varphi \in \Phi$ be an extension of $\psi$. We write $\varphi | \psi$.
Then $\pi$ acts semisimply on $H_{\psi}$ and has eigenvalues
$\varphi(\pi)$ for $\varphi | \psi$, each  with multiplicity $n^{2}$. 
Now $L = (\Lie X)_K$ is a quotient of $H$. It has the decomposition $L =
 \bigoplus L_{\psi}$. Assume first that $\psi \neq \psi_0$. By the
Eisenstein condition,  $\pi$ has on $L_{\psi}$ at most the eigenvalues
$\varphi(\pi)$ with $\varphi \in A_{\psi}$. The multiplicity of the
eigenvalues is at most  $n^{2}$. But since by (\ref{KC5e})
$\rank L_{\psi} = n^{2}a_{\psi}$, each of these eigenvalues must have
exactly multiplicity $n^2$. We assume now that $\psi = \psi_0$. Then
again the eigenspaces of $\varphi(\pi)$ have dimension $\leq n^{2}$. 
But the second of the Eisenstein conditions says that for $\varphi =
\varphi_0$ this multiplicity is $\leq n$. Since $\rank L_{\psi_{0}} =
a_{\psi}n^{2} + n$, this implies that the multiplicity of $\varphi(\pi)$
is $n^{2}$ for $\varphi \in A_{\psi_0}$ and is $n$ for $\varphi =
\varphi_0$. Altogether the multiplicity of the eigenvalue
$\varphi(\pi)$ of $\pi$ acting on $L$ is $r_{\varphi}n$. This implies
the Kottwitz condition on $L_K$. Since $({\bf K}_r)$ is a closed condition, the assertion follows.   
\end{proof}

We will now assume that $S$ is a $\kappa_{\tilde{E}}$-scheme. The
$\SuetO_D$-module $\SuetO_D \otimes \mathcal{O}_S$ defines a
polynomial function 
\begin{displaymath}
\rho: \mathbb{V}(\SuetO_D)_S \rightarrow \mathbb{A}^{1}_S. 
\end{displaymath} 
The module $\SuetO_D \otimes \mathcal{O}_S$ has a composition series
with factors 
\begin{displaymath}
\SuetO_D/\Pi \SuetO_D \otimes_{\tilde{\kappa}, \tilde{\psi}} \mathcal{O}_S ,
\end{displaymath}
where $\tilde{\psi} \in \tilde{\Psi}$.
This last module defines the polynomial function
$\chi_{\tilde{\psi}}$ from (\ref{KCchi1e}). One deduces easily that
the determinant of the module $\SuetO_D \otimes \mathcal{O}_S$ is 
\begin{displaymath}
\rho = \prod_{\tilde{\psi} \in \tilde{\Psi}} \chi_{\tilde{\psi}}^{ne} =
%\prod_{\psi \in \Psi} \chi_{\psi}^{e} = 
\prod_{\varphi \in \Phi} \Nrd_{\varphi}^{n}. 
\end{displaymath}
Analogously to (\ref{DC1e}) we have a decomposition
\begin{displaymath}
   \mathbb{D}(X)_{S} = \bigoplus_{\tilde{\psi} \in \tilde{\Psi}}
    \mathbb{D}(X)_{S,\tilde{\psi}},
\end{displaymath}
where all summands are locally free $\mathcal{O}_S$-modules of rank
$ne$. We set for $\psi \in \Psi$
\begin{displaymath}
   \mathbb{D}(X)_{S,\psi} = \bigoplus_{\tilde{\psi} \in \tilde{\Psi}}
   \mathbb{D}(X)_{S,\tilde{\psi}}, 
\end{displaymath}
where the sum runs over all $\tilde{\psi}$ such that
$\tilde{\psi}_{|F^{t}} = \psi$. 
It follows from Proposition \ref{DOD1p} that we have locally free
$\mathcal{O}_S$-modules with ranks
\begin{displaymath}
   \rank
   \mathbb{D}(X)_{S,\tilde{\psi}}/\pi^i\mathbb{D}(X)_{S,\tilde{\psi}}
   = ni, \quad \text{for} \; 0 \leq i \leq e. 
\end{displaymath}
We note that the $\mathcal{O}_S$-modules
$\mathbb{D}(X)_{S,\tilde{\psi}}$ are defined for each
$\SuetO_{\tilde E}$-scheme with $p$ locally nilpotent. 
\begin{proposition}\label{sfO2p}
Let $(X,\iota)$ be a $p$-divisible $\SuetO_D$-module over a
$\kappa_E$-scheme $S$.  There are natural surjective maps
\begin{equation}\label{KE1e}
   \mathbb{D}(X)_{S,\psi} \rightarrow \Lie_{\psi} X, \quad \text{for} \;
   \psi \in \Psi.
\end{equation}

\smallskip

\noindent (i) Assume that $(X,\iota)$ is $r$-special. 
Then the maps \eqref{KE1e} induce isomorphisms 

\begin{equation}\label{KE2e}
\begin{aligned}
\mathbb{D}(X)_{S,\psi}/\pi^{a_{\psi}}\mathbb{D}(X)_{S, \psi}
&\rightarrow \Lie_{\psi} X,  \text{ for }  \psi \neq \psi_0\\[1mm]
\mathbb{D}(X)_{S,\psi_0}/\pi^{a_{\psi_0}}\mathbb{D}(X)_{S, \psi_0}
&\rightarrow \Lie_{\psi_0} X / \pi^{a_{\psi_0}} \Lie_{\psi_0} X.  
\end{aligned}
\end{equation}

In particular, the cokernel of any power of $\pi$ on $\Lie X$ is a
locally free $\mathcal{O}_S$-module.

\smallskip

\noindent (ii) Conversely, assume that the following conditions on
$(X,\iota)$ are satisfied. 
 
\begin{enumerate}
\item $\Lie X$ satisfies the rank condition $(\mathbf{R}_r)$.
\item The natural map $\mathbb{D}(X)_S \rightarrow
\Lie X$ induces isomorphisms (\ref{KE2e})
\item $\Lie_{\psi_0}$ is annihilated by $\pi^{a_{\psi_0}+1}$.
\end{enumerate}
Then $X$ is $r$-special.
\end{proposition}

We will prove this together with the following Corollary:

\begin{corollary}\label{K_rR_r}
An $r$-special formal $\SuetO_D$-module $X$ over a
$\kappa_E$-scheme $S$ satisfies the Kottwitz condition $(\mathbf{K}_r)$.
\end{corollary}
\begin{remark}
In the case where $r=r^\circ$, this corollary follows from Lemma
\ref{KrDrDrin}.  
\end{remark}
\begin{proof} Clearly we can restrict to the case where $S$ is a
  scheme over $\kappa_{\tilde{E}}$. 

We first prove (i). 
The last condition of (\ref{Eisenstein2e}) says that for $\psi \neq
\psi_0$ the first 
arrow of (\ref{KE2e}) exists. By (\ref{KC5e}) we have on
both sides locally free $\mathcal{O}_S$-modules of the same
rank. Therefore this arrow is an isomorphism. 

For the second line  in \eqref{KE2e}, we begin with the case where $S = \Spec k$.
The second condition of (\ref{Eisenstein2e}) says that the rank of the
following homomorphism of vector spaces
\begin{displaymath}
   \pi^{a_{\psi_0}}: \Lie_{\psi_0} X \rightarrow \Lie_{\psi_0} X
\end{displaymath} 
is at most  $n$. This shows 
\begin{displaymath}
   \dim_k (\Lie_{\psi_0} X /\pi^{a_{\psi_0}} \Lie_{\psi_0} X) \geq \dim_k
   \Lie_{\psi_0}X \, - n = a_{\psi_0}n^2.
\end{displaymath}
Therefore the second arrow of (\ref{KE2e}) is an isomorphism because on
the left hand side we have a vector space of dimension
$n^2a_{\psi_0}$. 

It follows that the $\mathcal{O}_S$-module $\Lie_{\psi_0} X
/\pi^{a_{\psi_0}} \Lie_{\psi_0}$ has in each point of $S$ the same
rank $n^2a_{\psi_0}$. This already proves  assertion (i) in the case
where $S$ is a reduced scheme.

The general case is a consequence of 
 Lemma \ref{Eisen1l} below, applied to $L = \Lie_{\psi_0}X$, $f = \pi^{a_{\psi_0}}$,
$r = a_{\psi_0}n^{2}$, $m = a_{\psi_0}n^{2} + n$, and $s= n$.

Now we prove the corollary. We remark that the Kottwitz condition
$(\mathbf{K}_r)$ is 
satisfied for $\Lie X$. This is clear by the first isomorphism of
(\ref{KE2e}) for the part $\Lie_{\psi} X$, for $\psi\neq \psi_0$. By the second isomorphism
it suffices to show that the determinant morphism for the $\SuetO_D$-module $\pi^{a_{\psi_0}}
\Lie_{\psi_0} X$ is $\Nrd_{\varphi_0} = \prod_{\tilde{\psi}}
\chi_{\tilde{\psi}}$, where $\tilde{\psi}$ extends $\varphi_0$. But it
follows from the isomorphism (\ref{KE2e}) and (\ref{KC5e})  
that $\rank \pi^{a_{\psi_0}} \Lie_{\tilde{\psi}} X = 1$ for each
$\tilde{\psi} \in \tilde{\Psi}$ which extends $\varphi_0$. Therefore we
conclude by  Lemma \ref{KC1l}.    

Now we prove (ii).  We have to
prove the Eisenstein conditions. For $\psi \neq \psi_0$ they are
clear. We denote by ${K}_{\psi_0}$ the kernel of the natural map  
$\mathbb{D}(X)_{S,\psi_0} \rightarrow \Lie_{\psi_0} X$. We obtain
$\pi^{a_{\psi_0}}\mathbb{D}(X)_{S,\psi_0} \supset {K}_{\psi_0}
\supset \pi^{a_{\psi_0}+1}\mathbb{D}(X)_{S,\psi_0}$. Then the rank
condition shows that  
\begin{displaymath}
\pi^{a_{\psi_0}}\mathbb{D}(X)_{S,\psi_0}/{K}_{\psi_0} =
\pi^{a_{\psi_0}}\Lie_{\psi_0} X
\end{displaymath}
has rank $n$. This proves the second condition of
(\ref{Eisenstein2e}). The other Eisenstein conditions are trivially
satisfied. 
\end{proof}  

In the preceding proof, we used the following lemma. 
\begin{lemma}\label{Eisen1l}
Let $R$ be a local ring with residue  field $k$. Let $L$ be a
finitely generated free $R$-module of rank $m$. Let $f: L \rightarrow
L$ be an endomorphism. 

Let $r = \dim_k (L/f(L)) \otimes k$ and let $s = m -r$. 
We assume that 
\begin{displaymath}
\bigwedge^{s+1} f = 0.
\end{displaymath}
Then $L/f(L)$ is a free $R$-module of rank $r$. 
\end{lemma}
\begin{proof} The exact sequence 
\begin{displaymath}
L \otimes k \rightarrow L \otimes k \rightarrow (L/f(L)) \otimes k
\rightarrow 0
\end{displaymath}
shows that there is a basis of $L \otimes k$ of the form
\begin{equation}\label{Le1e}
f(\bar{y}_1), \ldots, f(\bar{y}_s), \bar{e}_1, \ldots, \bar{e}_r\, ,
\end{equation}
where $\bar{y}_1, \ldots, \bar{y}_s, \bar{e}_1, \ldots,
\bar{e}_r \in L \otimes k$ and the images of $\bar{e}_1, \ldots,
\bar{e}_r$ in $(L/f(L))\otimes k$ form a basis. 

Lifting the elements $\bar{y}_1, \ldots, \bar{y}_s, \bar{e}_1, \ldots,
\bar{e}_r$ to $L$ we obtain a basis of this $R$-module, 
\begin{displaymath}
f(y_1), \ldots, f(y_s), e_1, \ldots e_r. 
\end{displaymath}
The elements $\bar{y}_1, \ldots, \bar{y}_s \in L \otimes k$ are
linearly independent. Therefore we find a second basis of $L$, 
\begin{equation}\label{Le12e}
y_1, \ldots, y_s, x_1, \ldots, x_r. 
\end{equation}
We write the matrix of $f$ with respect to the basis (\ref{Le1e}) 
and (\ref{Le12e}) as a $(s+r)\times (s+r)$ block matrix:

\begin{displaymath}
      \left(
\begin{array}{cc}
E_s & A \\ 
\mathbf{0} & B
 \end{array}
 \right)
\end{displaymath}
The assumption $\bigwedge^{s+1} f = 0$ implies that the matrix $B$ is
zero. We find
\begin{displaymath}
f(x_i) = \sum_{j} a_{ji}f(y_j), \quad a_{ji} \in R.
\end{displaymath}
 
We set $z_i = x_i - \sum_{j} a_{ji} y_j \in \Ker f$. Clearly
\begin{displaymath}
y_1, \ldots, y_s, z_1, \ldots, z_r
\end{displaymath}
is also a basis of $L$, i.e.,  we may assume WLOG that $x_i = z_i$. But
then the matrix of $f$ becomes
\begin{displaymath}
      \left(
\begin{array}{cc}
E_s & \mathbf{0} \\ 
\mathbf{0} & \mathbf{0}
 \end{array}
 \right) .
\end{displaymath}
From this our assertion is obvious. 
\end{proof}
 
Before continuing,  we add a lemma needed later which is proved in the
same manner.

\begin{lemma}\label{Eisen2l}
Let $n$, $m$, and $r$ be natural numbers. 
Let $R$ be a commutative ring. Let $W$ be a locally free $R$-module of
rank $n$. Let $f: W \rightarrow W$ be an endomorphism such that $\Coker
f$ is a locally free $R$-module of rank $r$. 

\noindent Let $V \subset W$ be a direct summand of rank $m$. We assume that $s =
m- r \geq 0$ and that 
\begin{displaymath}
\bigwedge^{s+1}(f_{|V}) = 0.
\end{displaymath}  

\noindent Then $\Ker f \subset V$.
\end{lemma}
\begin{proof} We note that the assumptions of the lemma are compatible
with base change $R \rightarrow S$. The situation of the lemma is
always defined over a noetherian subring. Therefore we may assume that
$R$ is noetherian. The desired inclusion may be checked over the
localizations of $R$. Therefore we may assume that $R$ is a local
noetherian ring with maximal ideal $\mathfrak{m}$. Finally 
the matrix of $\Ker f \rightarrow W/V$ is zero if it is zero modulo
$\mathfrak{m}^t$ for all $t \in \mathbb{N}$. Therefore we may assume
that $R$ is an artinian local ring. We also note that under the
assumptions $\Ker f$ is a direct summand of $W$ of rank $r$. 

If $R$ is a field,  the assumption of the lemma implies $\rank f_{|V} \leq
s$. This implies $\dim \Ker f_{|V} \geq r = \dim \Ker f$. This implies
$\Ker f_{|V} = \Ker f$ and the lemma. 

Now let $(R, \ideal{m})$ be any artinian local ring.   Let $e_1,
\ldots, e_r$ be a basis of $\Ker 
f$. It follows from the case of a field that $V$ has a basis of the
form
\begin{equation}\label{L2e}
   v_1, \ldots, v_s, e_1 + \rho_1, \ldots, e_r + \rho_r,
\end{equation} 
where $\rho_i \in \ideal{m}W$. If the lemma is false we can choose $t
\in \mathbb{N}$ maximal such that there is a basis of the form
(\ref{L2e}) with $\rho_i \in \ideal{m}^tW$. 

By assumption $f(v_1), \ldots, f(v_s)$ are linearly independent
modulo $\ideal{m}$. Therefore we find a basis of $V$ of the form 
\begin{equation}\label{L21e}
   f(v_1), \ldots, f(v_s), u_1, \ldots, u_r.
\end{equation}  
We write the matrix of $f_{|V}$ with respect to the matrices
(\ref{L2e}) and (\ref{L21e}) as a block matrix, 

\begin{displaymath}
      \left(
\begin{array}{cc}
E_s & \ast \\ 
\mathbf{0} & X 
 \end{array}
 \right) .
\end{displaymath}
By the assumption $\wedge^{s+1}(f_{|V}) = 0$ all determinants of $(s+1)
\times (s+1)$ minors of this matrix are zero. Therefore the matrix $X$
is zero. We obtain equations
\begin{displaymath}
   f(\rho_i) = f(e_i + \rho_i) = \sum_{j} a_{ji} f(v_j).
\end{displaymath} 
Since the left hand side is in $\ideal{m}^tW$ we conclude that $a_{ji}
\in \ideal{m}^t$. We find $\rho_i - \sum_{j} a_{ji} v_j \in \Ker
f$. We may write 
\begin{displaymath}
   \rho_i - \sum_{j} a_{ji} v_j = \sum_k c_{ki}e_k,
\end{displaymath} 
where necessarily $c_{ki} \in \ideal{m}^t$. 
The last equation gives:
\begin{displaymath}
  (e_i + \rho_i) - \sum_{j} a_{ji} v_j - \sum_k c_{ki}(e_k +\rho_k) =
  e_i - \sum_k c_{ki}\rho_k.
\end{displaymath}
 The RHS is an element of $V$. But this shows that we may
 replace in (\ref{L2e}) the element $e_i + \rho_i$ by $e_i + \rho'_i$
 with $\rho'_i = \sum_k c_{ki}\rho_k \in \ideal{m}^{2t}W$. This is a
 contradiction. 
\end{proof}
\begin{remark} If $R$ is a local ring such that each
element of the maximal ideal $\mathfrak{m}$ is nilpotent, we can
replace the condition that ``$\Coker f$ is a free $R$-module of
rank $r$'' by the weaker assumption that ``$\Ker f \subset W$ is a
direct summand of rank $r$''. Indeed, over $R$ any free submodule of a
free module is a direct summand. This shows that $\Coker f$ is free. 
The weaker assumption also suffices if $R$ is a ring such that $R
\rightarrow \prod R_{\mathfrak{p}}$ is injective, where $\mathfrak{p}$
runs over all minimal prime ideals of $R$, since then we may reduce to
$R = R_{\mathfrak{p}}$. 
\end{remark}

We will now study the display of a formal $\SuetO_D$-module over
a $\kappa_E$-scheme $S$ which satisfies $(\mathbf{D}_r)$.   
To ease the notation we will assume that $S = \Spec R$. We denote by
$\mathcal{P} = (P,Q,F, \dot{F})$ the display of $X$. 
We use the notation $I := I_R := \ker (W(R)\to R)$. Recall that $\mathbb{D}(X)_R =
P/IP$. 

Again write $\Psi = \Hom(\kappa , \kappa_E)$ for the set of field
embeddings. We obtain the decompositions
\begin{displaymath}
 P = \bigoplus_{\psi \in \Psi} P_{\psi}, \quad  Q = \bigoplus_{\psi
   \in \Psi} Q_{\psi}. 
\end{displaymath}
By Proposition \ref{sfO2p} we obtain: 
\begin{equation}\label{Eisenstein6e}
\begin{aligned}
\pi^{a_{\psi_0}+1}P_{\psi_0} + IP_{\psi_0} \subset Q_{\psi_0} &\subset
\pi^{a_{\psi_0}}P_{\psi_0} + IP_{\psi_0},\\
Q_{\psi} &= \pi^{a_{\psi}}P_{\psi}  + IP_{\psi}, \quad \text{ for } 
\psi \neq \psi_0. 
\end{aligned}
\end{equation}
The maps $F$ and $\dot{F}$ induce maps 
\begin{displaymath}
F_{\psi}: P_{\psi} \rightarrow P_{\psi \sigma}, \quad \dot{F}_{\psi}:
P_{\psi} \rightarrow P_{\psi \sigma}.  
\end{displaymath}
Here $\sigma$ denotes the Frobenius automorphism of $F^{t}$ over
$\mathbb{Q}_p$. 
We set $Q'_{\psi} = P_{\psi}$ for $\psi \neq \psi_0$ and we define
$Q'_{\psi_0}$ to be the submodule such that $P_{\psi_0} \supset
Q'_{\psi_0} \supset Q_{\psi_0}$ and such that $Q'_{\psi_0}/Q_{\psi_0}$
is the kernel of the homomorphism
\begin{displaymath}
\pi^{a_{\psi_0}}: P_{\psi_0}/Q_{\psi_0} \rightarrow   P_{\psi_0}/Q_{\psi_0}.
\end{displaymath} 
By Lemma \ref{Eisen1l} we know that this kernel is a direct summand of
$P_{\psi_0}/Q_{\psi_0}$. 

We set 
\begin{displaymath}
Q' =  \bigoplus_{\psi} Q'_{\psi}, \quad F'_{\psi} = F_{\psi}\pi^{a_{\psi}},
\quad \dot{F}'_{\psi} = \dot{F}_{\psi}\pi^{a_{\psi}}. 
\end{displaymath}
We obtain Frobenius-linear homomorphisms
\begin{equation}\label{RamFunkt1e}
F' =  \bigoplus_{\psi} F_{\psi}\pi^{a_{\psi}}: P \rightarrow P, \quad 
\dot{F} =  \bigoplus_{\psi} \dot{F}_{\psi}\pi^{a_{\psi}}: Q' \rightarrow
P.   
\end{equation}
We claim that the quadruple $\mathcal{P}' = (P,Q',F', \dot{F}')$ is
the display of a special formal $\SuetO_D$-module.
\begin{theorem}\label{equivofcat}
Let $R$ be a $\kappa_E$-algebra. We assume that the nilradical of $R$
is a nilpotent ideal.
Let $\mathcal{C}_{r,R}$ be the category of $r$-special formal $\SuetO_D$-modules,  and let
$\mathcal{C}_{\mathbf{0},R}$ the category of special formal
$\SuetO_D$-modules (Definition \ref{ODmod1d}). 

The construction $\mathcal{P} \mapsto \mathcal{P}'$ is an equivalence
of categories
\begin{equation}\label{Eisenstein8e}
\mathcal{C}_{r,R} \rightarrow \mathcal{C}_{\mathbf{0},R}. 
\end{equation} 
\end{theorem} 

\begin{proof}
We begin with the case $S = \Spec k$, where $k$ is a perfect
field. Let the covariant Dieudonn\'e $M_X$ be identified with $P$.
In this case (\ref{Eisenstein6e}) is equivalent with 
\begin{equation}\label{Eisenstein6e2}
\begin{aligned}
\pi^{a_{\psi_0}+1}M_{\psi_0,X} \subset &VM_{\psi_0 \sigma,X} \subset 
\pi^{a_{\psi_0}}M_{\psi_0},\\
VM_{ \psi \sigma, X} &= \pi^{a_{\psi}} M_{\psi,X} \quad \text{ for }
\psi \neq \psi_0 .
\end{aligned}
\end{equation}
We define 
\begin{equation}\label{Eisenstein7e}
\begin{aligned}
V'= \pi^{-a_{\psi}}V&\colon M_{\psi \sigma,X} \rightarrow
M_{\psi,X}\\
F'= \pi^{a_{\psi}}F&\colon M_{\psi,X} \rightarrow M_{\psi\sigma,X} \, .
\end{aligned}
\end{equation}

Then the Dieudonn\'e module $(M_{X}, F',V')$ corresponds
to the display above. From this we see that $\mathcal{P}'$ is the Dieudonn\'e module of a special formal
$\SuetO_D$-module. Indeed, by the remark after (\ref{Eisenstein2e}) we
need only to verify $({\bf R}_r)$ (\ref{KC5e}) for $\mathcal{P}'$. But this follows
easily from (\ref{KC5e2}) and (\ref{Eisenstein6e2}).

If conversely $(M, F' , V')$ is the Dieudonn\'e module of a special
formal $\SuetO_D$-module, then we find 
\begin{displaymath}
F'M_{\psi_0,X} \subset \pi^{e-1}M_{\psi_0\sigma,X}, \quad
F'M_{\psi,X} \subset \pi^{e}M_{\psi\sigma,X}, \; \text{for} \; \psi \neq \psi_0.
\end{displaymath} 
This follows because $V'M_{\psi\sigma} = M_{\psi}$ for $\psi \neq
\psi_0$ and $M_{\psi_0}/V'M_{\psi_0\sigma}$ is annihilated by $\pi$. 
Therefore the formulas $V = \pi^{a_{\psi}}V'$ and $F =
\pi^{-a_{\psi}}F'$ define a Dieudonn\'e module structure on $M$ such
that (\ref{Eisenstein6e2}) is satisfied. This shows that $(M,F,V)$ is
the Dieudonn\'e module of an $r$-special  formal $\SuetO_D$-module. This proves the theorem in the case of a perfect
field.  

In the general case we need first to verify that $\mathcal{P}'$ is a
display. The only non-trivial property is that $\dot{F}'$ is a
Frobenius-linear epimorphism. To show this,  we take locally on $\Spec
R$ a  normal decomposition of $\mathcal{P}'$ and consider the matrix
of $F'  \bigoplus \dot{F}'$. We have to show that  the image of the
determinant of this matrix in $R$ is a unit. But this property
follows since we know it for a perfect field. The same argument
shows that $\mathcal{P}'$ is nilpotent. Therefore we have defined a
functor (\ref{Eisenstein8e}).  

We construct first a quasi-inverse functor in the case that the ring $R$ is
reduced. Let 
$\mathcal{P}'$ be the display of a special formal $\SuetO_D$-module.  
We note that $P'/IP'$ is a locally free $\SuetO_D \otimes R$ module of
rank $1$. In particular,  it has a filtration by direct summands as $R$-modules, 
\begin{displaymath}
0 = \pi^{e} (P'_{\psi}/IP'_{\psi}) \subset \pi^{e-1}
(P'_{\psi}/IP'_{\psi}) \subset \ldots \subset \pi 
(P'_{\psi}/IP'_{\psi}) \subset P'_{\psi}/IP'_{\psi} 
\end{displaymath} 
 The multiplication by $\pi$ gives
an isomorphism between the subquotients of this filtration. 

If we have a direct $R$-module summand $L \subset P'_{\psi}/IP'_{\psi}$
such that $\pi (P'_{\psi}/IP'_{\psi}) \subset L \subset
P'_{\psi}/IP'_{\psi}$,  we obtain therefore a direct $R$-module summand 
$\pi^{a_{\psi} + 1} (P'_{\psi}/IP'_{\psi}) \subset \pi^{a_{\psi}} L \subset
\pi^{a_{\psi}}P'_{\psi}/IP'_{\psi}$. 

This gives the possibility to invert our construction $\mathcal{P}
\rightarrow \mathcal{P}'$. We set $P = P'$. We note that  $Q'_{\psi} =
P'_{\psi}$ if $\psi \neq \psi_0$. We set in general 
\begin{equation}\label{sfr1e}
Q_{\psi} = \pi^{a_{\psi}}Q'_{\psi} + IP_{\psi}.
\end{equation}

We want to define $F$ and $\dot{F}$ by the formulas
\begin{displaymath}
F_{\psi} = \pi^{-a_{\psi}}F'_{\psi}, \quad
\dot{F}_{\psi} = \pi^{-a_{\psi}}\dot{F}'_{\psi}.
\end{displaymath}
We note that $F'_{\psi}Q'_{\psi} = p\dot{F}'_{\psi}Q'_{\psi}$. This implies for
$\psi \neq \psi_0$ that $F'_{\psi}P_{\psi} \subset
\pi^{e}P_{\psi\sigma}$. From $\pi P'_{\psi_0} \subset Q'_{\psi_0}$ we
conclude that $F'_{\psi_0}P_{\psi_0} \subset \pi^{e-1}P_{\psi_0\sigma}$.
Since  $R$ is reduced,  $\pi$ operates injectively on $W(R)$ and therefore
the definition of $F_{\psi}$ makes sense. From (\ref{sfr1e}) we see
that also the definition of $\dot{F}$ makes sense. We have to show
that $\mathcal{P} = (P,Q,F,\dot{F})$ is indeed a display. But this
follows from the case of a perfect field treated above. 

Now we treat the case of a nonreduced ring $R$. We assume that we
have a divided power thickening $R \rightarrow S$, and that the
theorem is already known for $S$. We denote by $X$ an $r$-special 
formal $\SuetO_D$-module over $S$,  and
by $X'$ the corresponding 
special formal $\SuetO_D$-module over $S$. We show that our functor
gives a bijection between the liftings of $X$ to an $r$-special  formal
$\SuetO_D$-module over $R$  and the
liftings of $X'$ to a special 
formal $\SuetO_D$-module over $R$. This will prove the theorem by induction. By Grothendieck-Messing, the
liftings of $X$ correspond to liftings of the Hodge-filtration,
\begin{equation}\label{sfGM1e}
\begin{CD}
\mathbb{D}(X)_{R,\psi} & @>>> & L_{\psi}\\
@VVV & & @VVV\\
\mathbb{D}(X)_{S,\psi} & @>>> & \Lie_{\psi} X \, .
\end{CD}
\end{equation}  
If $\psi \neq \psi_0$ we have no choice for $L_{\psi}$ because the
Proposition \ref{sfO2p} requires $L =
\mathbb{D}(X)_{R,\psi}/\pi^{a_{\psi}}\mathbb{D}(X)_{R,\psi}$. As a
special case this holds 
also for $X'$. Now let $\psi = \psi_0$. Let
$\bar{Q}_R$ and $\bar{Q}_S$ the kernels of the two horizontal maps in
(\ref{sfGM1e}). Then we have $\pi^{a_{\psi}}\mathbb{D}(X)_{R,\psi}
\supset \bar{Q}_R \supset \pi^{a_{\psi}+1}\mathbb{D}(X)_{R,\psi}$. 
We replace $\bar{Q}_R$ by $\bar{Q}'_R :=
\pi^{-a_{\psi}}\bar{Q}_R$. This makes sense because we have a bijection 
\begin{displaymath}
\mathbb{D}(X)_{R,\psi}/\pi\mathbb{D}(X)_{R,\psi}
\overset{\pi^{a_{\psi}}}{\longrightarrow}
\pi^{a_{\psi}}\mathbb{D}(X)_{R,\psi}/\pi^{a_{\psi}+1}\mathbb{D}(X)_{R,\psi}. 
\end{displaymath}
Then $L' = \mathbb{D}(X)_{R,\psi}/\bar{Q}'_R$ is a lifting of
$\Lie_{\psi} X'$ which defines a lifting of the special formal
$\SuetO_D$-module $X'$. This sets up the desired bijection of liftings. 
\end{proof}
\begin{corollary}\label{uniupto}
Let $k$ be an algebraically closed field which is at the same time a $\kappa_E$-algebra. Any two $r$-special formal $O_D$-modules over $k$ are isogenous by a $O_D$-linear isogeny. 
\end{corollary}
\begin{proof}
By Theorem \ref{equivofcat}, we are reduced to the case of special formal $O_D$-modules, i.e., the case $r=r^\circ$. In this case, the assertion  follows from \cite{D}, \S 2, comp. \cite{BC}, Prop.~5.2. 
\end{proof}

Let $\breve{E}$ the completion of the maximal unramified extension of
$E$. Its residue class field $\bar{k}$ is an algebraic closure of
$\kappa_E$. We fix an $r$-special  formal $O_D$-module $(\BX,
\iota_\BX)$ over $\bar{k}$ (a {\it framing object}). 
 
\begin{definition}\label{calMr}
 We  define the set-valued functor $\CM_r$ on the category
 of $O_{\breve{E}}$-schemes as follows\footnote{We will prove in Proposition \ref{identform} that this definition coincides with the one in section \ref{s:formu}.}.   Then $\CM_r$
 associates to scheme $S\in (Sch/ O_{\breve{E}})$ the set of
 isomorphism classes of triples $(X, \iota, \rho)$. Here $(X, \iota)$
 is an $r$-special formal $O_D$-module over 
 $S$, and $\rho: X \times_{\Spec O_{\breve{E}}} \Spec \bar{k} \to \BX\times_{\Spec
   \bar{k}} S$ is a $O_D$-linear  isogeny of height zero. 
\end{definition} 
We write $\overline{\CM_r}$ for the restriction of this functor to
$\bar{k}$-schemes $S$.  
Theorem \ref{equivofcat} now implies the following corollary. 
\begin{corollary}\label{specfibscheme}
The functor $\overline{\CM}_r$ is representable by a scheme over $\bar
k$ which is isomorphic to $\overline{\CM}_{r^\circ}$.   Hence there is
an isomorphism $\overline{\CM}_{r}\simeq
\hat{\Omega}^n\otimes_{O_{\breve E}} \bar k$.  
\end{corollary}
\begin{proof}
The isomorphism $\overline{\CM}_{r}\simeq\overline{\CM}_{r^\circ}$ follows from Theorem \ref{equivofcat}. The last assertion follows from \cite{D}, which also implies that $\overline{\CM}_{r^\circ}$ is a scheme. 
\end{proof}

\section{The local model }\label{s:localmod}

In this section we consider the local structure of the formal scheme
$\CM_r$ (Definition \ref{calMr}). By the general theory \cite{PRS},
this comes down to considering the {\it local model} of $\CM_r$.  Let
us define it.  

Recall that $D$ is the central divison algebra  with invariant $1/n$
over $F$. Let $V$ be a $D$-vector space of dimension $1$ and let
$\Lambda$ be an  $O_D$-lattice in $V$. 

The local model in question represents the following
functor on $(\rm{Sch}/O_E)$:
\begin{equation*}
\begin{aligned}
 \BM_r(S) = \{ \CF \subset \Lambda \otimes_{\BZ_p} \CO_S \mid  &\text{ $O_D$-stable   $\CO_S$-submodule, 
locally on $S$ a direct summand, }\\
&\text{ such that $ (\Lambda \otimes_{\BZ_p} \CO_S)/\CF$ satisfies conditions $({\bf R}_r)$ and $({\bf E}_r)$}\}.
 \end{aligned}
\end{equation*}

\begin{lemma}\label{genfiber}
The functor $\BM_r$ is representable by a projective scheme over $\Spec O_E$. The geometric generic fiber is isomorphic to $\BP^{n-1}$. 

Let $S$ be an $E$-scheme. Consider a $\CO_S$-submodule $\CF \subset
\Lambda \otimes_{\BZ_p} \CO_S$ which is locally a direct summand and 
which is $O_D$-stable. Then   $(\Lambda \otimes_{\BZ_p}
\CO_S)/\CF$ satisfies conditions $({\bf R}_r)$ and $({\bf E}_r)$ if and only if it satisfies the condition $({\bf K}_r)$.
\end{lemma}
\begin{proof} The first assertion is obvious since the rank condition
  is a closed condition, cf. the remark after Definition
  \ref{rkcond1d}.
  
   The implication $``\Rightarrow"$ in the last assertion follows as in Proposition \ref{flatimplies}. To show the converse, let
  $R$ be a $\bar \BQ_p$-algebra and $S=\Spec R$. Let $\CF$ be a direct
  summand of $\Lambda\otimes_{\BZ_p}R$  that is $O_D$-stable and such
  that $ (\Lambda \otimes_{\BZ_p} \CO_S)/\CF$ satisfies $({\bf
    K}_r)$. There are decompositions 
\begin{equation*}
\Lambda\otimes_{\BZ_p}R=\bigoplus\nolimits_\varphi \Lambda_\varphi\, ,\quad \CF=\bigoplus\nolimits_\varphi \CF_\varphi\, , 
\end{equation*}
where $\varphi$ runs through the embeddings of $F$ into $\bar \BQ_p$. 
Here $O_F$ acts on the summand corresponding to $\varphi$ via
$\varphi:F\to \bar\BQ_p\to R$. Each summand is stable under the action of
$D$. The condition  $({\bf K}_r)$  just says that
$\rank\,\Lambda_\varphi/ \CF_\varphi=nr_\varphi$, in which case
$\Lambda_\varphi/\CF_\varphi$ is locally on $S$ isomorphic to the
direct sum of $r_\varphi$ copies of  the simple representation
$F^n\otimes_\varphi R$  of $D\otimes_{F, \varphi} R\simeq {\rm
  M}_n(R)$. On the summand $\Lambda_\varphi/\CF_\varphi$, $\iota(\pi)$
acts as $\varphi(\pi)\,{\rm Id}_{\CF_\varphi}$. Let $\psi=\psi_0$. It
then follows that $Q_{A_{\psi_0}}(\iota(\pi))$ annihilates all
summands $\Lambda_\varphi/\CF_\varphi$, for those $\varphi$ with
$\varphi_{\vert_{F^t}}=\psi_0$ and  $\varphi\neq\varphi_0$, and
$Q_{A_{\psi_0}}(\iota(\pi))$ induces an isomorphism on
$\Lambda_{\varphi_0}/\CF_{\varphi_0}$, which implies the second
Eisenstein condition. The first and third Eisenstein conditions are
proved in an analogous way.  

For $\varphi$ with $\varphi_{\vert_{F^t}}\neq \psi_0$, the subspace
$\CF_\varphi$ is trivial, i.e., either equal to $(0)$ or to
$\Lambda_\varphi$. On the other hand, using  Morita equivalence, the
${\rm M}_n(R)$-stable summand $\CF_{\varphi_0}$ of
$\Lambda_{\varphi_0}$ corresponds to a hyperplane of
$F^n\otimes_{F,\varphi_0} R$.  It now follows that the geometric
generic fiber of $\BM_r$ is isomorphic to the projective space of
lines in $\bar \BQ_p^n$, i.e., to $\BP^{n-1}$ (Grothendieck's convention).  
\end{proof}

The geometric special fiber $\overline{\BM}_r=\BM_r\otimes_{O_E} \bar
k$ can be described as follows. Let $W_\psi = \Lambda \otimes
_{O_{F^t}, \psi} \bar k$, an $en^2$-dimensional vector space   
with its endomorphism $\Pi = \iota (\Pi)$. Let $S=\Spec\, R$, for a
$\bar k$-algebra $R$, and let $(\CF_\psi)_{\psi\in \Psi}$ be a point
in $\overline{\BM}_r(S)$. Let first $\psi\neq \psi_0$. By the third
Eisenstein condition, $\CF_\psi$ is a direct summand of rank
$(e-a_\psi) n^2$ containing the image  of $\Pi^{a_\psi n}$. Since
these two submodules are direct summands of the same rank, they are
equal.  

Now let $\psi=\psi_0$, and set $W_0=W_{\psi_0}$ and $a_0=a_{\psi_0}$. Then, due to the action of $O_{\tilde F}$, we obtain a $\BZ/n$-grading 
\begin{equation}
W_0=\bigoplus\nolimits_{k\in \BZ/n} W_{0, k},
\end{equation}
and $\Pi$ is an endomorphism of degree one. Forgetting the subspaces $\CF_\psi$ with $\psi\neq \psi_0$, we have an identification 
\begin{equation*}
\begin{aligned}
 \overline{\BM}_r (S)  = \{ \CF_0 \subset W_0 \otimes_{\bar k} &\CO_S \mid \Pi
\text{-stable    {\it graded} direct summand,}\\
&{\rm rank} \, (W_{0,k, S}/\CF_{0, k}) = a_0 n + 1 \, ,\forall k\in\BZ/n,
\text{ and $1')$ and $2')$}\}\, .
\end{aligned}
 \end{equation*}

Here we have set $W_{0, S}=W_0\otimes_{\bar k}\CO_S$ and $W_{0,k, S}=W_{0, k}\otimes_{\bar k}\CO_S$, and $1')$ and $2')$  are as follows: 
\begin{equation}\label{cond12}
\begin{aligned}
1')  &\quad \Pi^{(a_0+1)n}\vert ({  W_{0, S}/ \CF_0})  &= 0 \, \\
2') &\bigwedge^{n+1} (\Pi^{a_0n}{\vert (W_{0, S}/\CF_0}))&= 0 .
\end{aligned}
\end{equation}
Of course, we have used here \eqref{Eisenstein2e}.

Let us now apply  Lemma \ref{Eisen2l} to the $\CO_S$-dual $W_{0, S}^*$, its endomorphism $f$ induced by $(\Pi^*)^{a_0n}$ and its submodule $V=(W_{0, S}/\CF_0)^*$, in which case 
$\rank W_{0, S}^*=en^2$, and $\rank V=(a_0n+1)n$, and $r=a_0n^2$, and $s=n$. We conclude that $\Ker (\Pi^*)^{a_0n}\otimes_{\bar k}\CO_S\subset V$. Translated back into $\CF_0$, we obtain a chain of inclusions of direct summands of $W_{0, S}$,
\begin{equation}
\Im (\Pi^{(a_0+1)n})\otimes_{\bar k}\CO_S\subset \CF_0\subset \Im (\Pi^{a_0n})\otimes_{\bar k}\CO_S \, .
\end{equation}
In the Drinfeld case $r = r^{\circ}$ we have $a_0 = 0$. Let us write  $
\BM^\circ$ for $
\BM_{r^\circ}$. 

% Recall the local model $\BM^\circ$ for the Drinfeld moduli problem,
% i.e. $\BM^\circ=\BM_{r^\circ}$. It represents the functor which to
% an $O_F$-algebra $R$ associates the set of $O_D$-stable direct
% summands $\CG$ of $\Lambda\otimes_{O_F} R$ such that  
% \begin{equation*}
% {\rm char} (x \big\vert(\Lambda\otimes_{O_F} R)/\CG) = {\rm chard}
% (x) , \quad \forall x \in O_D\, .  
% \end{equation*}

Let us identify $\Im\Pi^{(a_0+1)n}/\Im\Pi^{a_0n}$ with
$W_0/\Pi W_0$. Associating now to an
$S$-valued point $\CF_0$ of  
 $\overline{\BM}_r$ the locally direct summand 
$$\CF_0/\Im (\Pi^{(a_0+1)n})\otimes_{\bar k}\CO_S\subset  (\Im \Pi^{a_0 n}/\Im \Pi^{(a_0+1)n})\otimes_{\bar k}\CO_S=(\Lambda\otimes_{O_F, \varphi_0}\bar k)\otimes_{\bar k}\CO_S\, ,
$$ we have obtained an $S$-valued point of the local model $
\BM^\circ$, more precisely of $ \BM^\circ\otimes_{O_F,
  \varphi_0}O_{E}\otimes_{O_{E}}\bar k$. Letting $S$
vary, this induces obviously an isomorphism of schemes over $\bar k$,  
\begin{equation}\label{isospfiber}
\BM_r\otimes_{O_E}\bar k\simeq  \BM^\circ\otimes_{O_F}\bar k . 
\end{equation}
Therefore we obtain from the Drinfeld case:
\begin{corollary}\label{locmodred}
 The geometric special fiber $\BM_r\otimes_{O_E}\bar k$ is a reduced
 scheme, which has $n$ irreducible components, all of which have
 dimension $n-1$. Furthermore, the local rings of closed points of
 $\BM_r\otimes_{O_E}\bar k$ are isomorphic to localizations in closed
 points of the $\bar k$-algebra $\bar
 k[X_1,\ldots,X_n]/(X_1\cdot\ldots\cdot X_n)$. \qed 
\end{corollary}
\begin{remark}
We point out that $\BM^\circ$ coincides with the {\it standard local model} for the triple $(G, \{\mu\}, K)$ consisting of $\GL_n$ and $\mu_{(1, 0^{(n-1)})}$ and the Iwahori subgroup, after extension of scalars to $\Spec O_{\breve E}$, cf. \cite{G}, comp. also \cite{PR.I, PRS}. 
\end{remark}
\begin{corollary} \label{locmodflat}
$\BM_r$ is flat over $O_E$.
\end{corollary}
\begin{proof} Since the special fiber is reduced, it suffices by \cite{GW},
 Prop.~14.16 to show that a generic point of an irreducible component
 of the special fiber is in the closure of the general fibre. 
Since the general fiber and the special fibre of $\BM_r$ have the same
dimension $n-1$ and since $\BM_r$ is proper it follows that at least
one irreducible component of the special fiber is contained in the
closure of the generic fiber. The claim therefore follows from the
following lemma. 
\end{proof}
\begin{lemma} There is an action of $\BZ/n$ on $\BM_r\otimes_{O_E}O_{\breve E}$, which induces a  transitive action on the set of irreducible components of  $\BM_r\otimes_{O_E}\bar k$. 
\end{lemma}
\begin{proof} The action is given by sending $\CF=\oplus_{\tilde\psi\in\tilde\Psi}\CF_{\tilde\psi}$ to $\CF'$ with 
$$
 (\CF')_{\tilde\psi}=\CF_{\tilde\psi\tau} \, ,\quad \tilde\psi\in\tilde\Psi\, . 
$$
Here $\tau$ is taken from the presentation of $O_D$ in \eqref{OD1e}. Indeed, since the $O_F$-structure of $(\Lambda \otimes_{\BZ_p} \CO_S)/\CF'$ coincides with that of $(\Lambda \otimes_{\BZ_p} \CO_S)/\CF$, the Eisenstein conditions are satisfied for $(\Lambda \otimes_{\BZ_p} \CO_S)/\CF'$, since they are satisfied for $(\Lambda \otimes_{\BZ_p} \CO_S)/\CF$. On the other hand, for any $\psi\in\Psi$, 
$$
\rank (\CF')_{\psi} = \sum_{\tilde\psi \in \tilde\Psi_{\psi}} \rank (\CF')_{\tilde\psi}=\sum_{\tilde\psi \in \tilde\Psi_{\psi}} \rank \CF_{\tilde\psi\tau}=\sum_{\tilde\psi \in \tilde\Psi_{\psi}} \rank \CF_{\tilde\psi}=r_\psi .
$$
Therefore, $\CF'$ also satisfies the condition $({\bf R}_r)$. 

That the action of $\BZ/n$  on the set of irreducible components of  $\BM_r\otimes_{O_E}\bar k$ is transitive, follows from the corresponding fact for $ \BM^\circ$ (the Drinfeld case) (the isomorphism \eqref{isospfiber} is obviously equivariant for the action of $\BZ/n$). 
\end{proof}
\begin{corollary}
The scheme $\BM_r$ is normal. 
\end{corollary}
\begin{proof}
Indeed, $\BM_r$ is flat over $O_E$, with  normal generic fiber  (even regular, cf. Lemma \ref{genfiber}), and reduced special fiber. These properties imply that $\BM_r$ is normal, cf. \cite{PZ}, Prop. 9.2. 
\end{proof}
 
 \begin{remark} If the Conjecture \ref{conjDr} were true, it would follow that $\BM_r$ has semi-stable reduction, in particular $\BM_r$ would be regular. However, we are unable to prove these stronger assertions. 
 \end{remark}

Let $\CM_r'$ be the closed formal  subscheme of $\CM_r$ which is given by the
Kottwitz condition $(\mathbf{K}_r)$. 
 % We now apply the previous results to the formal scheme $\CM_r$. More
 % precisely, we define two formal schemes, $\CM_r$ and $\CM'_r$, over
 % $\Spf {O_{\breve E}}$. We do this by prescribing the functors on the
 % category ${\rm Nilp}_{O_{\breve E}}$ they represent. We fix a framing
 % object $(\BX, \iota_\BX)$ over $\bar k$, cf. Corollary
 % \ref{uniupto}. Both functors associate to $S\in {\rm Nilp}_{O_{\breve
 %     E}}$ the set of isomorphism classes of triples $(X, \iota, \rho)$
 % where $(X, \iota)$ is a formal $O_D$-module over $S$ and where $\rho:
 % X\times_S\bar S\to \BX\times_{\Spec \bar k} \bar S$ is a $O_D$-linear
 % quasi-isogeny of height zero. For $\CM_r$ we impose on $(X, \iota)$
 % the conditions $({\bf K}_r)$ and $({\bf E}_r)$. For $\CM'_r$ we
 % impose on $(X, \iota)$ the conditions $({\bf R}_r)$ and $({\bf
 %   E}_r)$. By Lemma \ref{KC0l}, the formal scheme $\CM_r$ is a closed
 % formal subscheme of $\CM'_r$, and
 By Corollary \ref{K_rR_r} the
 special fibers of $\CM_r$ and $\CM'_r$ are identical.
 \begin{proposition}\label{identform}
   The two formal schemes $\CM_r$ and $\CM'_r$ are identical. Both are
   $p$-adic and flat over $\Spf {O_{\breve E}}$, with special fiber
   $\CM_r\times_{\Spf O_{\breve E}}\Spec \bar k$ a reduced scheme. All
   their completed local rings are normal.
 \end{proposition}
 \begin{proof}
   We use the {\it local model diagram}
  $$
  \xymatrix{
    & {\wt\CM}_r \ar[dl]_-{\varphi} \ar[dr]^{\psi}\\
    \CM_r & & \widehat{\BM}_r \, , }
$$ 
where $\wh{\BM}_r$ denotes the formal completion of
$\BM_r\times_{\Spec O_E}\Spec O_{\breve E}$ along its special
fiber. Here ${\wt\CM}_r$ and the morphism $\varphi$ is obtained from
$\CM_r$ by adding to $(X, \iota, \rho)$ an
$O_D\otimes_{\BZ_p}\CO_S$-linear isomorphism with the value at $S$ of
the covariant crystal associated to $X$,
$$
\alpha\colon \Lambda\otimes_{\BZ_p}\CO_S\lra\BD(X)_S \, .
$$ 
The morphism $\psi$ associates to $(X, \iota, \rho, \alpha)$ the
submodule $\CF=\alpha^{-1}(\Ker(\BD(X)\to\Lie X))$ of
$\Lambda\otimes_{\BZ_p}\CO_S$. The theory of local models \cite{RZ}
tells us that the completed local ring of a point $x\in\CM_r$ is
isomorphic to the completed local ring of $\psi(\tilde x)$, where
$\tilde x$ is any point of ${\wt\CM}_r$ mapping under $\varphi$ to
$x$. Hence all completed local rings of points of $\CM_r$ are
isomorphic to completed rings of points of $\BM_r\times_{\Spec
  O_E}\Spec O_{\breve E}$. Hence by Corollaries \ref{locmodflat} and
\ref{locmodred}, the formal scheme $\CM_r$ is flat over $\Spf
O_{\breve E}$ with all completed local rings normal. That
$\CM_r\times_{\Spf O_{\breve E}}\Spec \bar k$ is a reduced scheme
follows from Corollary \ref{specfibscheme}. Now the equality of
$\CM'_r$ and $\CM_r$ follows from Proposition \ref{flatimplies}.
\end{proof}

\begin{corollary}\label{indepofpi}
The definition of $\CM_r$ is independent of the choice of the uniformizer $\pi$ of $O_F$.
\end{corollary}
\begin{proof} Consider the formal scheme $\CN_r$ that represents the moduli problem where the Kottwitz condition $({\bf K}_r)$ is imposed but the Eisenstein conditions $({\bf E}_r)$ are dropped.  Let $\tilde\pi$ be another uniformizer, and let $\tilde\CM_r$ be the corresponding formal scheme defined using the Eisenstein condition for $\tilde\pi$ instead of $\pi$. What has to be shown is that the formal subschemes $\CM_r$ and $\tilde\CM_r$ of $\CN$ are identical. Let $\BN_r$ be the local model corresponding to $\CN_r$; then the local models $\BM_r$ and $\tilde\BM_r$ of $\CM_r$ and $\tilde\CM_r$ are closed subschemes of $\BN_r$. It suffices to prove that $\BM_r=\tilde\BM_r$. But  by Lemma \ref{genfiber} the generic fibers of $\BN_r, \BM_r$ and $\tilde\BM_r$ all coincide and,  by  Corollary \ref{locmodflat},   $\BM_r$ and $\tilde\BM_r$ are equal to the flat closure of the generic fiber inside $\BN_r$. 

\end{proof}

\section{The generic fiber (after Scholze)}\label{s:genfiber}
In this section, we prove the last point in Theorem \ref{mainDR}, in the following form. For convenience, we introduce for a function $r\colon\varphi\mapsto r_\varphi$ the formal scheme $\tilde{\CM}_r$  over $\Spf O_{\breve{E}}$ that represents the same moduli problem as $\CM_r$, except that we drop the condition that the height of $\rho$ be zero. Then our original formal scheme 
$\CM_r$ is an open and closed formal subscheme of $\tilde{\CM}_r$.

Let $r^{\circ}$ be the Drinfeld  function, i.~e., $r^{\circ}_{\varphi} = 0, \forall \varphi \neq \varphi_0$. We write $\tilde{\CM}^{\circ} = \tilde{\CM}_{r^{\circ}}$.

We will prove the following theorem. We use the embedding $\breve F\hookrightarrow \breve E$ defined by the  natural map $\breve F=F\otimes_{F^t}\breve \BQ_p\overset{\varphi_0\otimes\id}{\longrightarrow} E\otimes_{E^t}\breve \BQ_p=\breve E$. 

\begin{theorem}\label{maingenfib}
 There is an isomorphism of adic spaces over $\rm{Spa} (\breve{E}, O_{\breve{E}})$,
 \begin{equation*}
  (\tilde{\CM}_r)^{ad} \simeq (\tilde{\CM}^{\circ}
  \hat{\otimes}_{O_{\breve{F}}} O_{\breve{E}})^{ad} 
 \end{equation*}
\end{theorem}
This theorem implies the last point in Theorem \ref{mainDR}. Indeed, passing to the open and closed sublocus where the
universal quasi-isogeny $\varrho$ has height zero, we obtain a similar isomorphism when $\tilde{\CM}_r$ is replaced by
$\CM_r$ and $\tilde{\CM}^{\circ}$ by $\CM^{\circ}$ (the proof of Theorem \ref{maingenfib} will show that the isomorphism in question
is compatible with the decompositions according to the height). Since by Drinfeld's theorem $(\CM^{\circ})^{ad} \simeq
\Omega^n_F$, we deduce the desired isomorphism
\begin{equation*}
 (\CM_r)^{ad} \simeq \hat\Omega_F \otimes_F \breve{E} = (\hat{\Omega}_F \otimes_{O_F} O_{\breve{E}})^{ad} \, .
\end{equation*}
In the proof of Theorem \ref{maingenfib}, we will use the following notation. We denote by $(\BX, \iota_\BX)$ the framing object for the moduli
problem $\tilde{\CM}_r$. Let $M(\BX)_{\BQ_p}$ be its rational Diendonn\'{e} module.

Let $V$ be a free $D$-module of rank one, and let $\Lambda$ be an
$O_D$-lattice in $V$. Let $G = \GL_D (V)$, considered as a
linear algebraic group over $\BQ_p$. The function $r$ defines a
$G$-homogeneous projective variety $\Fb$ over $E$. If $R$ is a
$\bar{\BQ}_p$-algebra, then $\Fb(R)$ parametrizes $D$-linear
surjective homomorphisms into locally free $R$-modules
\begin{equation*}
  V \otimes_{\BQ_p} R \lra \CF
\end{equation*}
such that under the decomposition $\CF = \bigoplus_{\varphi}
\CF_{\varphi}$, we have $\rank (\CF_{\varphi}) = r_{\varphi} n$. In other words, $\Fb$ coincides with the generic fiber of the local model, $\BM_r\times_{\Spec O_E}\Spec E$, cf. Lemma \ref{genfiber}. 

Let $K_0 \subset G(\BQ_p)$ be the stabilizer of the lattice $
\Lambda$. For any open subgroup $K \subset K_0$, we obtain the
corresponding member $\BM_K$ of the RZ-tower over
$(\tilde{\CM}_r)^{ad}$. These coverings of $(\tilde{\CM}_r)^{ad}$
parametrize level-$K$-structures on the universal object
$X/\tilde{\CM}_r$,
\begin{equation*}
  \alpha :  \Lambda \lra T(X)\,\, \rm{mod} \, K \, .
\end{equation*}
We denote by
\begin{equation*}
  \pi_K : \BM_K \lra (\CF \otimes_E \breve{E})^{ad}
\end{equation*}
the crystalline period maps, which are compatible with changes in $K$.

In \cite{SW} Scholze and Weinstein define a preperfectoid space $\BM_\infty$ by imposing a full level structure on $T(X)$.
In particular, there is a morphism
\begin{equation*}
 \BM_{\infty} \lra \underleftarrow{\lim} \, \BM_K \, ,
\end{equation*}
which induces a bijection for any algebraically closed extension $C$ of $\bar{\BQ}_p$ which is complete for a $p$-adic valuation,
\begin{equation*}
 \BM_{\infty} (C) \simeq \underleftarrow{\lim} \, \BM_K (C) \, .
\end{equation*}
We denote by $\pi_{\infty} : \BM_{\infty} \lra (\Fb \otimes_E \breve{E})^{ad}$ the induced period mapping. In \cite{SW},   the following
description of $\BM_{\infty}(C)$ is given. Let $\CO_C$ be the ring of integers in $C$.

Let $B^+_{\rm cris} = A_{\rm cris} (\CO_C / p) \otimes_{\BZ_p} \BQ_p$ be Fontaine's ring attached to $C$. The Fargues-Fontaine
curve\footnote{The Fargues-Fontaine curve is usually denoted by $X$; since this notation is already in use for the universal $p$-divisible group, we use the notation $Y$ instead.} is defined as $Y = {\rm Proj} \, P$, where $P$ is the graded ring
\begin{equation*}
 P = \bigoplus_{d \geq 0} (B^+_{\rm cris})^{\phi = p^d} \, .
\end{equation*}
Then $Y$ is a connected separated regular noetherian scheme of dimension $1$, equipped with the point $\infty \in Y$ corresponding
to Fontaine's homomorphism
\begin{equation*}
 \theta : B^+_{\rm cris} \lra C \, .
\end{equation*}
In \cite{SW} appears a description of $p$-divisible groups $X$ over $\CO_C$, in terms of two vector bundles, $\CE$ and $\CF$. Here
\begin{itemize}
\item $\CF = T \otimes_{\BZ_p} \CO_Y,$ 

\item  $\CE$ corresponds to the graded $P$-module $\bigoplus_{d \geq 0} (M_{\BQ_p})^{\phi = p^{d+1}}$.
\end{itemize}
Here $T = T (X)$ denotes the Tate module of the generic fibre of $X$, and $M_{\BQ_p} = M(X)_{\BQ_p}$ the rational Dieudonn\'{e} module of the 
reduction modulo $p$ of $X$.

We now apply this description to the fibers of the universal $p$-divisible group $X$ at points of $\BM_{\infty} (C)$,
noting that the universal full level structure induces an isomorphism $T(X) =  \Lambda$, and the universal quasi-isogeny 
an isomorphism of rational Dieudonn\'e modules $M(X)_{\BQ_p} = M(\BX)_{\BQ_p} \otimes_{\breve{\BQ}_p} B^+_{\rm cris}$. Accordingly we set 
$\CF =  \Lambda \otimes_{\BZ_p} \CO_Y$, and let $\CE$ correspond to the graded module
\begin{equation*}
 \big(\bigoplus_{d \geq 0} M (\BX)_{\BQ_p}  \otimes_{\breve{\BQ}_p} B^+_{\rm cris}\big)^{\phi = p^{d+1}} \, .
\end{equation*}
We also fix a $D$-linear isomorphism $M(\BX)_{\BQ_p} = V \otimes_{\BQ_p} {\breve{\BQ}_p}$.
The Scholze-Weinstein description therefore implies the following fact, cf. \cite{SW}, Cor. 6.3.10 (and its extension to the EL-case in Thm. 6.5.4).
\begin{theorem} There is an identification of $\BM_{\infty} (C)$ with the set of injective $D$-linear 
 homomorphisms of vector bundles on $Y$, 
\begin{equation*}
 f : \CF \lra \CE
\end{equation*}
such that ${\rm supp} ( {\Coker} f) = \{ \infty \}$ and such that $\mathfrak{m}_{Y, \infty}$ kills ${\Coker} f$ and such that the induced surjective map
\begin{equation*}
 \CE \otimes_{\CO_{X, \infty}} C = M (\BX)_{\BQ_p} \otimes_{\breve{\BQ}_p} C = V \otimes_{\BQ_p} C \, \buildrel{\varphi_f}\over\longrightarrow \, {\rm Coker} \, f
\end{equation*}
defines a point in $\Fb (C)$.
Furthermore, the period morphism $\pi_{\infty}$ sends the point corresponding to $f$ to $[ V \otimes_{\BQ_p} C \, \buildrel{\varphi_f}\over\longrightarrow \, {\Coker} f ]$.
\end{theorem}
It will be convenient to reformulate this last description.
Let $Y_F = Y \times_{\Spec\BQ_p} \Spec F$. This is a finite \'{e}tale cover $\psi : Y_F \lra Y$ of degree $d$, and the fiber $\psi^{-1}(\{ \infty \})$ can be
identified with $\{ \infty_\varphi \mid \varphi : F \lra \bar{\BQ}_p \}$. Since the vector bundles $\CF$ and $\CE$ are equipped with $F$-actions,
they are of the form $\CF = \psi_* (\CF_F)$ and $\CE = \psi_* (\CE_F)$, with vector bundles $\CF_F$ and $\CE_F$ over $Y_F$.
\begin{corollary}
 There is a natural identification of $\BM_{\infty} (C)$ with the set of $D$-linear injective morphisms of vector bundles on $Y_F$, 
\begin{equation*}
 f_F : \CF_F \lra \CE_F \, ,
\end{equation*}
such that ${\rm supp} \, {\Coker} f_F \subset \psi^{-1} (\{\infty\})$, with ${\Coker} f_F$ killed by $\mathfrak{m}_{Y_F, \infty_\varphi} \, \forall \varphi$,
and such that
\begin{equation*}
\dim_C ({\Coker} f_F)_{\infty_\varphi} = r_{\varphi} \cdot n \, , \, \forall \varphi \, .
\end{equation*}
Furthermore, the period morphism $\pi_{\infty}$ sends the point corresponding to $f_F$ to the family of surjections $[\CE_F \otimes_{Y_F, \infty_\varphi} C 
= V \otimes_{F, \varphi} C \lra ({\Coker} f_F)_{\infty_\varphi}]$, considered as a point in $\Fb (C)$.
\end{corollary}
Now we compare the previous descriptions for the given function $r$,
and for the Drinfeld function $r^{\circ}$. Let $\BX^{\circ}$ denote  
the framing object for $\tilde{\CM}_{r^{\circ}}$. Then we may choose $D$-linear isomorphisms
\begin{equation*}
 V \otimes_{\BQ_p} \breve{\BQ}_p = M (\BX^{\circ})_{\BQ_p}  = V \otimes_{\BQ_p} \breve{\BQ}_p= M (\BX)_{\BQ_p}
\end{equation*}
such that for the respective Frobenius endomorphisms
\begin{equation*}
 F = \pi^m \cdot F^{\circ} \, ,
\end{equation*}
where
\begin{equation*}
 m = \# \{ \varphi \mid r_{\varphi} = n \} \, .
\end{equation*}
For the corresponding vector bundles on $Y_F$, we get
\begin{equation*}
 \CE_F = \CE^{\circ}_F \otimes_{\CO_{X_F}} \CL^m \, ,
\end{equation*}
where $\CL$ is the line bundle on $Y_F$ corresponding to the graded module
\begin{equation*}
 \bigoplus_{d \geq 0} (F \otimes_{\BQ_p} B^+_{\rm cris})^{\phi = \pi p^d} \, .
\end{equation*}
We have a natural identification
\begin{equation*}
 \CF_F =  \Lambda \otimes_{O_F} \CO_{Y_F} = \CF^{\circ}_F \, .
\end{equation*}
On the other hand, we have a natural identification
\begin{equation*}
 \Fb = \Fb^{\circ} \otimes_F E \, .
\end{equation*}
Here, for a $\bar{\BQ}_p$-algebra $R$, a point $[V \otimes_{\BQ_p} R \lra \CF] = [ \{ V \otimes_{F, \varphi} R \lra \CF_{\varphi} \mid \varphi \} ]$
of $\Fb$ is sent to $[ \{ V \otimes_{F, \varphi} R \lra \CF^{\circ}_{\varphi} \} ]$, with 
\begin{equation*}
 \CF^{\circ}_{\varphi}= 
\begin{cases}
\begin{array}{ll}
\CF_{\varphi_0}&  \text{if} \; \varphi = \varphi_0\\
V \otimes_{F, \varphi} R & \text{if} \; \varphi \neq \varphi_0\, .
\end{array}
\end{cases}
\end{equation*}
The general case of an $E$-algebra $R$ follows by descent. We point out that, by Morita equivalence,
$\Fb^{\circ} \cong \BP^{n -1}_F$.
\begin{lemma}
 For a given $\varphi$, there exists a global section
\begin{equation*}
 LT_{\varphi} \in \Gamma (Y_F, \CL) = (F \otimes_{\BQ_p} B^+_{\rm cris})^{\phi= \pi}
\end{equation*}
such that $LT_\varphi$ vanishes to first order at $\infty_\varphi$ and is non-vanishing at all other points. Furthermore, $LT_{\varphi}$ is 
unique up to $F^\times$.
\end{lemma}
\begin{proof}
 The homomorphism $LT_{\varphi} : \CO_{Y_F} \lra \CL$ corresponds in the Scholze-Weinstein description to the Lubin-Tate formal
group corresponding to $(F, \varphi, \pi)$.
\end{proof}
\begin{lemma}
 Let $i : \CG_F \lra \CE_F$ be the injection of vector bundles on $Y_F$,
\begin{equation*}
 \CG_F = \{ x \in \CE_F \mid x \equiv 0 \mod \infty_{\varphi} \, , \, \forall \varphi \,\, {\rm with} \, r_{\varphi} = n \} \, .
\end{equation*}
Then the map $x \longmapsto x \cdot \Pi_{ \{ \varphi \mid r_{\varphi} = n \} } LT_{\varphi}$ defines an 
isomorphism $\CE^{\circ}_F \simeq \CG_F$.
\end{lemma}
\begin{proof}
 The map 
 $$
 x\mapsto x\prod\nolimits_{\{\varphi\mid r_\varphi =n\}} LT_\varphi
 $$ obviously identifies $\CE^{\circ}_F$ with a subbundle of $\CG_F$. Comparing the degrees of $\CE^{\circ}_F$
and $\CG$, the result follows. (There is a formalism of degrees of vector bundles on $Y$ resembling the usual theory over smooth
projective curves.)
\end{proof}
\begin{proposition}
 Under the identification $\CE^{\circ}_F = \CG_F$ and $\CF_F=\CF^{\circ}_F$, the map sending $f^{\circ}_F\in \Hom(\CF^\circ_F, \CE^\circ_F)$ to 
 $f_F = i \circ f^{\circ}_F\in \Hom(\CF_F, \CE_F)$ defines a
bijection $\BM^{\circ}_{\infty} (C) = \BM_{\infty} (C)$ which commutes with the period map $\pi^{\circ}_{\infty}$, resp.
$\pi_{\infty}$, to $\Fb^{\circ} (C) = \Fb (C)$.  
\end{proposition}
\begin{proof}
 It is clear that $\Coker f^{\circ}_F$ has support in $\{ \infty_{\varphi_0} \} \cup \{ \infty_{\varphi} \mid r_{\varphi} = n \}$,
and that for $\varphi$ with $r_{\varphi} = n$, one has that $(\Coker f^{\circ}_F)_{\infty_F} = \CE_F \otimes_{\CO_{Y_F}, \infty_{\varphi}} C$.
Conversely, any $f_F \in \BM_{\infty} (C)$ has to factor through $\CG_F$ and has the correct cokernel at $\infty_{\varphi_0}$. The
assertion regarding the period map is obvious from the way that $\Fb (C)$ is identified with $\Fb^{\circ} (C)$.
\end{proof}
\begin{corollary}\label{imper}
 Under the identification $\Fb^{ad} = (\Fb^{\circ} \otimes_F E)^{ad}$, the images of the period morphisms $\pi_{\infty}$ and $\pi^{\circ}_{\infty} \otimes_{\breve{F}} \breve{E}$,
coincide.
\end{corollary}
\begin{proof}
 All maps are partially proper \cite{F}, hence it suffices to prove for all algebraically closed complete extensions $C$ of $\bar{\BQ}_p$ that
\begin{equation*}
 \Im \, \pi_{\infty} (C) = \Im \, \pi^{\circ}_{\infty} (C) \, .
\end{equation*}
This follows from the previous proposition. 
\end{proof}
\begin{proof}[Proof of Theorem \ref{maingenfib}]
We need to construct an isomorphism
\begin{equation*}
 \BM_{K_0} \simeq \BM^{\circ}_{K_0} \otimes_{\breve{F}} \breve{E} \, .
\end{equation*}
It suffices to construct the isomorphism on the open and closed subloci $\BM^{(n)}_{K_0}$, resp. $\BM^{\circ(n)}_{K_0}$,
where the height of $\varrho$ is a fixed integer $n$. But the fibers of the period morphisms $\pi_{K_0}$, resp. $\pi^{\circ}_{K_0}$,
through $\BM^{(n)}_{K_0}$, resp. $\BM^{\circ(n)}_{K_0}$,  can both be identified with $G (\BQ_p)^{\circ} / K_0$, where under the
identification $G (\BQ_p) = D^{\times}$, we have 
\begin{equation*}
 G (\BQ_p)^{\circ} = \{ x \in D^{\times} \mid \ord \det x = 0 \} \, .
\end{equation*}
Since $K_0 = G (\BQ_p)^{\circ}$, the period maps identify $\BM^{(n)}_{K_0}$ and $\BM^{\circ (n)}_{K_0}$ with open adic subsets
of $(\Fb \otimes_E \breve{E})^{ad} = (\Fb^{\circ} \otimes_F \breve{E})^{ad}$.
The assertion therefore follows from Corollary \ref{imper}.
\end{proof}

\section{The unramified case}\label{s:unramified}
In this section we prove Theorem \ref{mainUR}. Hence in this section $F/\BQ_p$ is unramified. Since we fixed an embedding
$\varphi_0 : F \lra \bar\BQ_p$, we may identify
\begin{equation*}
 \Hom_{\BQ_p} (F, \bar\BQ_p) = \{ \sigma^i \mid i \in \BZ/d \} \, .
\end{equation*}
In particular, $E = F$, via $\varphi_0$. We abbreviate
\begin{equation*}
 r_i = r_{\varphi_0 \circ \sigma ^i}  \, , \quad i\in \BZ/d \, .
\end{equation*}
Let $S$ be a $O_F$-scheme, such that $p$ is nilpotent on $S$.  We have
a decomposition 
\begin{equation*}
 O_F \otimes_{\BZ_p} \CO_S = \bigoplus_{i \in \BZ/d} \CO_S \, ,
\end{equation*}
where the action of $O_F$ on the $i$-th summand is via $\sigma^i$.  If
$(X, \iota)$ is a formal $O_D$-module over $S$, we correspondingly
obtain a decomposition,
\begin{equation*}
  \Lie X = \bigoplus_{i \in \BZ /d} \Lie\!_i \, X \, ,
\end{equation*}
where $t \in \Lie\!_i \, X$ iff $\iota(x)( t) = \sigma^i (x) \, t$,
$\forall x \in O_F$.
\begin{lemma}
  The condition $({\bf R}_r)$ implies the conditions $({\bf K}_r)$ and
  $({\bf E}_r)$.  
\end{lemma}
\begin{proof}
For the implication $({\bf R}_r)\Rightarrow ({\bf E}_r)$, we refer
to Remarks \ref{eisrem}, (iii). The condition $({\bf D}_r)$ implies
$({\bf K}_r)$ by Proposition \ref{identform}. 
\end{proof}

We will now consider the display of $X$. To simplify the notation, we
assume that $S=\Spec R$, where $R$ is an $O_F$-algebra.  Let $W(O_F)$
be the ring of Witt vectors of $O_F$, and define a ring homomorphism
by
\begin{equation*}
  \lambda : O_F \lra W(O_F) \, , \quad w_m (\lambda (x)) = \sigma^m (x)\, .
\end{equation*}
Then $\lambda$ is Frobenius equivariant, i.e., $\lambda (\sigma (x)) =
^F\!\!(\lambda(x))$. For a $O_F$-algebra $R$, we obtain
\begin{equation*}
  \bar{\lambda} : O_F \lra W (O_F) \lra W (R)\, .
\end{equation*}
Let
\begin{equation*}
  \bar{\lambda}^{(i)} = \bar{\lambda} \circ \sigma^i = ^{F^i}\!\!
  \circ \bar{\lambda} : O_F \lra W(R). 
\end{equation*}

Let $\CP = (P, Q, F, \dot{F})$ be the display of $X$ over $R$. Hence
$P$ is a finitely generated projective $W(R)$-module, 
and $Q$ is submodule of $P$, and
\begin{equation*}
 F : P \lra P \, ,\,\,  \dot{F} : Q \lra P.
\end{equation*}
The action of $\SuetO_F $ on $\CP$ defines 
decompositions, 
\begin{equation*}
 P = \bigoplus P_i \, , \quad Q = \bigoplus Q_i \, ,
\end{equation*}
where $\SuetO_F$ acts on the $i$-th summand via
$\bar{\lambda}^{(i)}$. With respect to this $\BZ/d$-grading, the
operators 
$F$ and $\dot{F}$ are of degree $1$,
\begin{equation*}
 F : P_i \lra P_{i+1} \, , \quad \dot{F} : Q_i \lra P_{i+1} \, .
\end{equation*}
Then rank $P_i = n^2$, $\forall i,$ by Proposition \ref{DOD1p}. The
rank condition says:
\begin{equation}
\begin{aligned}
 P_0/Q_0 \, {\it is \, a \, locally \, free\,} R{-\it module \, of \,
   rank\, } n\\ 
P_i = Q_i \,\, {\it if} \, r_i = 0 \, ,\, \, I_R \cdot P_i = Q_i \,\,
{\it if} \, r_i = n \, . 
 \end{aligned}
\end{equation}
We choose a normal decomposition,
\begin{equation*}
 P_i = T_i \oplus L_i \, , \text{for all}\; i \, .
\end{equation*}
Then
\begin{equation*}
 F \oplus \dot{F} : T_i \oplus L_i \lra P_{i+1}
\end{equation*}
is a $^F$-linear isomorphism, for all $i$.
Let $i \neq 0$. When $r_i = 0$, then $P_i = Q_i$ and we obtain an isomorphism
\begin{equation}\label{iso0}
 \dot{F}^{\#} : W(R) \otimes_{F, W(R)} P_i \lra P_{i+1} \, .
\end{equation}
When $r_i = n$, then $L_i = (0)$, and we obtain an isomorphism
\begin{equation}\label{ison}
 F^{\#} : W(R)\otimes_{F, W(R)} P_i \lra P_{i+1} \, .
\end{equation}
We now define $^{F^d}$-linear homomorphisms,
\begin{equation*}
  F^{\#}_{\rm rel} : P^{(p^d)}_0 \lra P_0 \, , \, \,\dot{F}^{\#}_{\rm rel} : Q^{(p^d)}_0 \lra P_0 \, ,
\end{equation*}
as compositions,
\begin{equation*}
  F^{\#}_{\rm rel} : P^{(p^d)}_0 \overset{F^{\#}}\lra P_1^{(p^{d-1})} \tilde{\lra} \ldots \tilde{\lra} P_{d-1}^{(p)} \tilde{\lra} P_0 \, ,
\end{equation*}
resp.
\begin{equation*}
  \dot{F}^{\#}_{\rm rel} : Q^{(p^d)}_0 \lra P_1^{(p^{d-1})} \tilde{\lra} \ldots \tilde{\lra} P_{d-1}^{(p)} \tilde{\lra} P_0 \, .
\end{equation*}
Here the isomorphisms in the last two lines are either \eqref{iso0} or
\eqref{ison}.  We note that $P_0$ is a finitely generated projective
$W(R)$-module, that $Q_0$ is a submodule and $F_{\rm rel}$ and
$\dot{F}_{\rm rel}$ satisfy the following relations,
\begin{equation}\label{propoffrel}
  \begin{aligned}
    \dot{F}_{\rm rel} (^{V} \xi x) & = ^{F^{d-1}}\!\! \xi \cdot F_{\rm
      rel} (x) , & x \in P_0 \, .\\ 
p \cdot \dot{F}_{\rm rel} (y) & = F_{\rm rel} (y), & y \in Q_0 \, .
  \end{aligned}
\end{equation}
Indeed, the first identity reflects the $F^d$-linearity of
$\dot{F}_{\rm rel}$; the second identity comes from the fact that a
similar identity holds for $\dot{F}$ and $F$. The quadruple $(P_0,
Q_0, F_{\rm rel}, \dot{F}_{\rm rel})$ is a $d$-display in the
following sense:

\begin{definition}
Let $d\geq 1$ be a natural number. 
Let $R$ be a ring such that $p$ is a nilpotent in $R$. An {\it $d$-display
over $R$} is a quadruple $(P, Q, F, \dot{F})$, where $P$ is a finitely
generated projective $W(R)$-module, $Q$ a submodule of $P$ and
$F:P\to Q$ and $\dot{F}:Q\to P$ are $~^{F^{d}}$-linear maps such that the
following properties are satisfied: 
\begin{enumerate}
\item $I_{R}P\subset Q$, and $P/Q$ is a direct summand of the $R$-module $P/I_{R} P$. 
\item The linearization of $\dot{F}$,
$$
\dot{F}^\sharp: W(R)\otimes_{F^{d}, W(R)} Q \to P
$$
is surjective. 
\item For $x\in P$ and $w\in W(R)$, 
$$
\dot{F}(^V\!\!w x)=^{F^{d-1}}\!\!\!\!w F(x) . 
$$
\end{enumerate}
\end{definition}
We will now use the theory of relative Witt vectors and relative
displays. We denote by $q = p^{d}$ the number of elements in the
residue class field $\kappa$ of $\SuetO_F$. Also, when we use $\SuetO_F$ as a subscript, we simply write $\SuetO$.  
For an $\SuetO_F$-algebra $R$, we denote  by $W_{\!O}(R)$ the ring of
relative Witt vectors defined by the Witt polynomials 
\begin{equation*}
 w'_n(x_0, \ldots, x_n) = x_0^{q^n} + p  x_1^{q^{n-1}} + \cdots +
 p^n x_n.  
\end{equation*}
We have a canonical morphism $u: W(R) \rightarrow  W_{\!O}(R)$ such
that  
\begin{equation*}
\begin{aligned}
 u (^{F^d} \xi) & = ^{F'}\!\! \!u (\xi)\\
u (^V \xi) & = ^{V'}\!\! (u(^{F^{d-1}} \xi)) \, ,
\end{aligned}
\end{equation*}
cf. \cite{D}, Prop. 1.2. Here we denoted by a prime the operators on
$ W_{\!\suetO} (R)$. 

We now show how to associate to $(P_0, Q_0, F_{\rm rel}, \dot{F}_{\rm
  rel})$ a relative display $(P', Q', F', \dot{F}')$  with respect to
$ W_{\!O}(R)$ (replace the Witt vectors by the relative Witt vectors in the definition of a display).  

We set 
\begin{equation*}
 \begin{aligned}
  P' = &  W_{\!O} (R) \otimes_{u, W(R)} P_0  \\
Q' = & \Ker \big( W_{\!O} (R) \otimes_{u, W(R)} P_0 \lra P_0 / Q_0\big)\, .
 \end{aligned}
\end{equation*}
Here the last homomorphism is given by the composition
\begin{equation*}
 \begin{aligned}
 W_{\!O} (R) \otimes_{u, W(R)} P_0 \lra &  W_{\!O} (R) / I_{O}
(R) \otimes_{u, W(R)} P_0 = R \otimes_{u, W(R)} P_0 \\ 
& = P_0/ I (R) P_0 \lra P_0/Q_0 \, .  
 \end{aligned}
\end{equation*}
\begin{equation*}
 F' = ^{F'}\!\! \otimes F_{\rm rel} : P' \lra P' \, .
\end{equation*}
Note that this makes sense because of \eqref{propoffrel}, and defines a $^{F^d}$-linear endomorphism of $P'$. It remains to define
$\dot{F}' : Q' \lra P'$ with the following properties,
\begin{equation}\label{defoff}
 \begin{aligned}
  \dot{F}' (^{V'}\!\xi x) & = \xi \cdot F'(x), & x \in P' \\
\dot{F}' (\xi \otimes y) & = ^{F'}\!\! \xi \otimes \dot{F}_{\rm red} (y), & y \in Q \, .
 \end{aligned}
\end{equation}
More precisely, consider the normal decomposition $P_0 = T_0 \oplus L_0$. Then
\begin{equation*}
 Q' = \big(I_{O} (R) \otimes_{W(R)} T_0\big) \oplus  \big(W_{\!O} (R)
 \otimes_{W(R)} L_0\big) \, . 
\end{equation*}
We define $\dot{F}'$ on the first, resp. second summand by 
\begin{equation*}
 \begin{aligned}
\dot{F}' (^{V'}\!\! \xi \otimes t_0) & = \xi \otimes F_{\rm rel} (t_0)
\, , && t_0 \in T_0 \\ 
  \dot{F}' ( \xi \otimes l_0) & = ^{F'}\!\!\xi \otimes F_{\rm rel} (l_0) \, , && l_0 \in L_0 \, .
 \end{aligned}
\end{equation*}
{\bf Claim}: {\it The identities \eqref{defoff} are satisfied.}

\medskip

We start with the second identity. Let
\begin{equation*}
 y = ^{V'}\! \!\!\eta \cdot t_0 + l_0 \quad , \, t_0 \in T_0,\,  l_0 \in L_0 \, .
\end{equation*}
For the second summand, the identity to be checked is the definition of $\dot{F}'$. So we may
take $l_0 = 0$. Now
\begin{equation*}
 \begin{aligned}
\xi \otimes ^{V'}\!\! \eta \cdot t_0 & = \xi \cdot u (^{V'}\!\! \eta) \otimes t_0 = \xi \cdot ^{V'}\!\! u(^{F^{d-1}} \eta) \otimes t_0 \\  
& = ^{V'}\!\! (^{F'}\!\! \xi \cdot u (^{F^{d-1}} \eta)) \otimes t_0 \, .
 \end{aligned}
\end{equation*}
Hence the LHS of the identity to be checked is
\begin{equation*}
 \begin{aligned}
\dot{F}' (\xi \otimes \, ^{V'}\!\! \eta \, t_0) & = ^{F'}\!\! \xi u (^{F^{d-1}}\!\! \eta) \otimes F_{\rm rel} (t_0) \\
& =   ^{F'}\!\! \xi \otimes ^{F^{d-1}}\!\! \eta \cdot F_{\rm rel} (t_0) \\
& =  ^{F'}\!\! \xi \otimes \dot{F}_{\rm rel} (^{V'}\!\! \eta\, t_0) \, ,
 \end{aligned}
\end{equation*}
where in the last equation we used \eqref{propoffrel}. The second identity of \eqref{defoff} is proved.

Now we check the first identity. It suffices to check that 
\begin{equation*}
 \dot{F}' (^{V'}\! \xi \otimes x) = \xi \otimes F_{\rm rel} (x) \, ,
 \quad x \in P_0\, . 
\end{equation*}
If $x = t_0 \in T_0$, this holds by definition. Let $x = l_0 \in L_0$.
Then
\begin{equation*}
 \begin{aligned}
\dot{F}' (^{V'}\!\xi \otimes l_0) & = p \xi \otimes  \dot{F}_{\rm rel} (l_0) = \xi \otimes p \dot{F}_{\rm rel} (l_0) \\
& = \xi \otimes F_{\rm rel} (l_0) \, ,
 \end{aligned} 
\end{equation*}
where in the last equation we used \eqref{propoffrel}.

We now have checked that $(P', Q', F', \dot{F}')$ is a relative
display relative to $\SuetO_F$. By functoriality this relative display
has an $\SuetO_D$-action. Its Lie algebra $P'/Q'$ coincides with $\Lie_0
X = P_0/Q_0$. Therefore the action of $\SuetO_D$ on $P'/Q'$ is special in the
sense of Drinfeld (satisfies condition $({\bf R}_{r^\circ})$), i.e.,
we are in the case of Proposition  \ref{KC1p}. 
This implies automatically that the relative display is nilpotent. 

Let $R$ be an $\SuetO_F$-algebra where $\pi R$ is a nilpotent ideal.
By a theorem of Ahsendorf \cite{A}, Thm.~5.3.8, there is an
equivalence of categories between the category of $p$-divisible formal
$\SuetO_F$-modules over $R$ and the category of nilpotent relative
displays (this holds even without the hypothesis that $F/\BQ_p$ is
unramified). We therefore obtain a formal $\SuetO_F$-module $X'$ over
$S = \Spec R$, which is a special formal $\SuetO_D$-module because
$\Lie X' = P'/Q'$.

Applying the above construction to the framing object $(\BX,
\iota_{\BX})$, we obtain a special formal $\SuetO_D$-module $(\BX',
\iota_{\BX'})$ that we use as a framing object for the Drinfeld
functor $\CM_{r^\circ}$. Since the above construction is functorial in
$S$, we obtain a morphism of formal schemes over $\Spf O_{\breve F}$,
\begin{equation}\label{functormorph}
  \CM_r \lra \CM_{r^\circ} \, .
\end{equation}

\begin{theorem} The  morphism \eqref{functormorph} is an
isomorphism. In particular, there is an isomorphism of formal  
 schemes over $\Spf O_{\breve{F}}$,

\begin{equation*}
 \CM_r \simeq \hat\Omega^n_F \hat{\otimes}_{O_F} \SuetO_{\breve{F}} \, .
\end{equation*}
\end{theorem}
\begin{proof}
Let $R$ be a $\kappa_F$-algebra. We have (compare (\ref{RamFunkt1e}))
described a functor $\mathcal{P} \mapsto \mathcal{P}'$ which
associates to the display of a formal $O_D$-module with condition
$(\mathbf{D}_r)$ the display of a special formal $O_D$-module with
condition $(\mathbf{D}_{r^{\circ}})$. In the unramified case we take
$\pi = p$. Since $a_{\psi} = 0$  or $ 1$, we see that 
\begin{displaymath}
\dot{F}'_{\psi} = \dot{F}_{\psi}, \; or \; \dot{F}'_{\psi} =
p\dot{F}_{\psi} = F_{\psi}.
\end{displaymath}
Therefore starting from $\mathcal{P}$ or $\mathcal{P}'$ the
construction (\ref{propoffrel}) yields the same $d$-display. Therefore over
$\kappa_F$ the functor morphism (\ref{functormorph}) factors
\begin{displaymath}
  \CM_r \otimes_{O_F} \kappa_F \lra \CM_{r^\circ} \otimes_{O_F}
  \kappa_F \lra \CM_{r^\circ} \otimes_{O_F} \kappa_F,
\end{displaymath}
The first morphism is given by Theorem \ref{equivofcat} and the second
morphism is defined by associating to a display of a formal
$O_D$-module satisfying $(\mathbf{D}_{r^{\circ}})$ the display relative
to $O_F$. Therfore the last morphism is the identity. Since the first
arrow is an isomorphism by Theorem \ref{equivofcat}, we deduce that
(\ref{functormorph}) is an isomorphism over $\kappa_F$. But since both
schemes are flat by Proposition \ref{identform},  it is an isomorphism. 
\end{proof}

The only thing we use of Ahsendorf's work is: to each relative display
$\mathcal{P}'$ there is a formal group $BT (\mathcal{P}')$. 
Let $X$ be a formal $O_D$-module satisfying
$(\mathbf{D}_{r^{\circ}})$. Let $\mathcal{P}$ be the display of $X$. Let
$\mathcal{P}'$ be the relative display given by construction
(\ref{defoff}). Then $BT (\mathcal{P}')$ is canonically isomorphic
to $X$. 

%  We show that the construction that leads from $(X, \iota, \varrho)$ to
%  $(X', \iota', \varrho')$ can be reversed. To see this, start with $(X', \iota', \varrho')$ over $R$.  Let $\CP'=(P', Q', F', \dot{F'})$ be the display of the underlying $p$-divisible formal group of $X'$. Then the action of $\suetO_F$ on $P'$ and $Q'$ induces decompositions of $W(R)$-modules, 
% \begin{equation}
% P'=\bigoplus\nolimits_{i\in\BZ/d}P'_i,\,  Q'=\bigoplus\nolimits_{i\in\BZ/d}Q'_i, \quad Q'_i=Q'\cap P'_i ,
% \end{equation}
% such that $\deg F'=\deg \dot{F'}=1$. Furthermore, $P'_i=Q'_i$ for $i\neq 0$. We now define a new display $\CP=(P, Q, F, \dot{F})$ by setting $ P_i=P'_i$ for all $i$, and $Q_0=Q'_0$ and, for $i\neq 0$, 
% \begin{equation*}
% Q_i = 
% \begin{cases}
% \begin{array}{ll}
% P_i &  \text{if} \; r_i=0\\
% I_R P_i&\; \text{if} \; r_i=n\, ,
% \end{array}
% \end{cases}
% \end{equation*}
% and defining new $^F$-linear operators of degree one $F, \dot{F}$ by reversing the rules \eqref{iso0} and \eqref{ison}. {\bf IST DAS RICHTIG??} The $p$-divisible group $X$ corresponding to the display $\CP$ is equipped with an action $\iota$ of $O_D$,  and  is the  $r$-special formal $O_D$-module that is mapped under \eqref{functormorph} to $(X', \iota')$. 

\section{The Lubin-Tate moduli problem}\label{s:LT}
In this section we sketch that a modification of our method is
applicable to a variant of the Lubin-Tate moduli problem. 
Let again $F$ be an extension of degree $d$ of $\BQ_p$. We again fix
an embedding $\varphi_0 : F \lra \bar{\BQ}_p$. We 
also fix an integer $n \geq 2$, and function $r : \Phi \lra \BZ_{\geq
  0}$ with the same properties 
as in \eqref{propr}. Let $E$ be the corresponding reflex field.

Let $S$ be an $O_E$-scheme. Let $(\CL, \iota)$ be a locally free
$\CO_S$-module of finite rank with an action by $O_F$. We say that
{\it $(\CL, \iota)$ is an $O_F$-module over $S$}.  

Let $\tilde E$ be the normal closure of $E$. Then each $\varphi \in
\Phi$ factors through $\varphi : O_F \rightarrow O_{\tilde{E}}$. The
induced homomorphism $O_F \otimes_{\mathbb{Z}_p} O_{\tilde{E}} \rightarrow
O_{\tilde{E}}$ defines a map
\begin{displaymath}
{\rm Nm}_{\varphi}: \BV(O_F)_{O_{\tilde{E}}}\to \BA^1_{O_{\tilde{E}}}.
\end{displaymath}
We set ${\rm Nm}_r = \prod\nolimits_{\varphi}{\rm
  Nm}_\varphi^{r_\varphi}$, comp. \eqref{normmorph}. This is a
polynomial function defined over $O_E$,
$$
{\rm Nm}_r: \BV(O_F)_{O_E}\to \BA^1_{O_E} .
$$
We use the notations $\Psi$, $a_{\psi}$ from (\ref{abpsi1e}). For an
$O_E$-module $\mathcal{L}$ we have a natural decomposition 
\begin{displaymath}
\mathcal{L} = \bigoplus_{\psi \in \Psi} \mathcal{L}_{\psi}.
\end{displaymath}
We say that $(\mathcal{L}, \iota)$ satifies the rank
 condition $(\mathbf{R}_r^{F})$ if 
\begin{equation}
\rank_{\mathcal{O}_S} \mathcal{L}_{\psi_0} = a_{\psi_0}n + 1, \quad
\rank_{\mathcal{O}_S} \mathcal{L}_{\psi} = a_{\psi}n, \; \text{for} \,
\psi \neq \psi_0. 
\end{equation}
This implies $\rank_{\mathcal{O}_S} \mathcal{L} = \sum_{\varphi \in
  \Phi} r_{\varphi}$. 
With the notations (\ref{Qpsi1e}) we introduce the Eisenstein
conditions $(\mathbf{E}_r^{F})$ (compare Definition \ref{Eisen1d}):

\begin{equation}\label{EisensteinFe}
\begin{array}{rrrr}
((Q_0 \cdot Q_{A_{\psi_0}})(\iota(\pi)) | \mathcal{L}_{\psi_0}) & = & 0, &\\[2mm]
\bigwedge^{2}(Q_{A_{\psi_0}}(\iota(\pi) | \mathcal{L}_{\psi_0})) & = & 0, &\\[2mm]
(Q_{A_{\psi}}(\iota(\pi)) | \mathcal{L}_{\psi}) & = & 0, & \text{for} \;
\psi \neq \psi_0.\\[3mm]
\end{array}
\end{equation}
We say that $\mathcal{L}$ satisfies $(\mathbf{LT}_r^{F})$ if
$(\mathbf{R}_r^{F})$ and $(\mathbf{E}_r^{F})$ are satisfied. 

We say that an $O_F$-module $(\CL, \iota)$ over an $O_E$-scheme $S$ {\it satisfies the Kottwitz condition $({\bf K}^F_r)$ } if
\begin{equation}
{\rm Nm}_\CL={\rm Nm}_r \, ,
\end{equation}
an equality of two morphisms from $\BV(O_F)\times_{\Spec \BZ_p}S$ to $ \BA^1_S$. The condition is equivalent to the identity of polynomials with coefficients in $\CO_S$, 
\begin{equation}\label{kottwitzLT}
 {\rm char} \, (\iota (x) \mid \CL) = \prod\nolimits_{\varphi} (T - \varphi (x))^{r_\varphi} \, , \,\, \forall x \in O_F \, ,
\end{equation}
cf. Remark \ref{kottchar}.

We will consider $p$-divisible groups $X$ of height $n d$ over
$O_E$-schemes $S$ with an action $\iota : O_F \lra \End (X)$. We
assume that the action on $\Lie X$ satisfies the condition
$(\mathbf{LT}_r^F)$. We will also say that $X$ is of type $r$.  

Let us assume that $S =
\Spec k$ is the spectrum of a perfect field of characteristic $p$. Let
$M$ be the Dieudonn\'e module of $X$. We obtain a
decomposition
\begin{displaymath}
M = \bigoplus_{\psi \in \Psi} M_{\psi}, 
\end{displaymath}  
where $M_{\psi}$ is a free $W(k)$-module of rank $n [F:F^{t}]$. 
By the condition $(\mathbf{LT}_r^F)$ we find (compare Proposition
\ref{sfO2p}):  
\begin{displaymath}
\begin{array}{cc}
\pi^{a_{\psi_0} + 1} M_{\psi_0} \subset VM_{\sigma \psi_0} \subset
\pi^{a_{\psi_0}} M_{\psi_0}, & \\[2mm]
VM_{\sigma \psi} = \pi^{a_{\psi}}M_{\psi}, & \text{for} \; \psi \neq \psi_0. 
\end{array}
\end{displaymath}
If we replace $(M,F,V)$ by $(M,\oplus \pi^{a_{\psi}}F, \oplus
\pi^{-a_{\psi}}V)$, we obtain a Dieudonn\'e module of a $p$-divisible
group $Y$ of type $r^{\circ}$ for the Drinfeld function $r^{\circ}$.
We see that $X$ is a formal $p$-divisible group if and only if $Y$ is. Then
$(Y, \iota)$ is of Lubin-Tate type, cf \cite{D}, i.e., $Y$ is a strict formal
$O_F$-module of dimension one and $O_F$-height $n$. We conclude that
over an algebraically closed field $\bar{k}$ there is up to isomorphism a
unique formal $p$-divisible  group $(X,\iota)$ of type $r$. 

In general we say that $(X, \iota)$ over a scheme $S$ is a {\it
Lubin-Tate group of type r} if it is a formal $p$-divisible group of
height $n d$ such that the action of $O_{F}$ on $\Lie X$ satisfies
$(\mathbf{LT}_r^F)$. 
\begin{remark}
In the case $r=r^\circ$, the Eisenstein conditions  $(\mathbf{E}_r^{F})$ in $(\mathbf{LT}_r^F)$ are redundant. This follows from the flatness of the {\it naive local model} in \cite{PR.I}, in the case where $d=n, r=1$ in the notation of loc.~cit. But it is also easy to see this directly. 
\end{remark}
From the case $r = r^{\circ}$ we deduce: 
\begin{proposition}\label{uniqueLT}
Any Lubin-Tate group $(X,\iota)$ of type $r$ over $\bar k$  is
isoclinic. Any two  
Lubin-Tate groups of type $r$ over $\bar k$ are isomorphic. Any
$O_F$-linear  quasi-isogeny of height zero between Lubin-Tate groups
of type $r$ is an isomorphism. 
\end{proposition}

More generally the proof of Theorem \ref{equivofcat} shows the following fact. 
\begin{proposition} Let $S$ be a $\bar k$-scheme. Then there is an
  equivalence between the category of Lubin-Tate groups of type $r$ 
  and the category  of Lubin-Tate groups of type $r^{\circ}$. \qed
\end{proposition}

We fix a Lubin-Tate group  $(\BX, \iota_{\BX})$ of type $r$ over the residue class field $\bar
k$  
of $\bar{\mathbb{Q}}_p$. It is unique up to
isomorphism. We may therefore define a functor $ \CM^F_r$ on ${\rm
  Nilp}_{O_{\breve{{E}}}}$. Namely, fixing a framing object $(\BX,
\iota_\BX)$ over $\bar k$,  $ \CM^F_r$ associates to $S\in {\rm
  Nilp}_{O_{\breve{{E}}}}$ the set of isomorphism classes of triples
$(X, \iota, \rho)$ where $(X, \iota)$ is a Lubin-Tate group of type
$r$ over $S$ and 
$\rho: X\times_S\bar S\to \BX\times_{\Spec \bar k}\bar S$ is an
$O_F$-linear quasi-isogeny of height zero. The 
formal scheme representing this functor will be denoted by the same
symbol. 

The main theorem in this section is the following.
\begin{theorem}\label{LTtheorem}
 The formal scheme $\CM^F_r$ is isomorphic to $\Spf O_{\breve{E}}
 [[t_1, \cdots , t_{n-1}]]$. In particular a Lubin-Tate group of type
 $r$ satisfies the Kottwitz condition $({\bf K}^F_r)$. 
\end{theorem}
We will show now how this follows from the results of previous
sections. The Kottwitz condition follows as in the proof of
Proposition \ref{flatimplies}. 
We first note the following consequence of Corollary
\ref{uniqueLT}.  
\begin{corollary}
 $\CM^F_r (\bar k) = \{ {\rm pt} \}$, hence $\CM^F_r$ is the formal spectrum of a local $O_{\breve{E}}$-algebra $R$ w.r.t. its maximal ideal.\qed
\end{corollary}
In order to complete the proof of Theorem \ref{LTtheorem}, it remains
to show that $\CM^F_r$ is formally smooth of relative dimension $n-1$
over $O_{\breve{E}}$. This will follow from the theory of local models.  

Let $V$ be a $F$-vector space of dimension $n$ and $\Lambda$ an
$O_F$-lattice in $V$. Let $\BM^F_r$ be the functor on $({\rm Sch} \, / O_E)$ such that 
\begin{equation*}
\begin{aligned}
\BM^F_r (S) = \{ \CF \subset \Lambda \otimes_{\BZ_p} \CO_S \mid O_F\text{ -stable
direct summand such that }\\
\text{ $(\Lambda \otimes_{{\BZ_p}} \CO_S)/\CF$ satisfies $({\bf R}^F_r)$ and $({\bf E}_r^F)$}\} .
\end{aligned}
\end{equation*}

Then $R$ is isomorphic to the completion of $\BM^F_r$ at a point of $\BM^F_r (\bar k)$. Hence Theorem \ref{LTtheorem} follows from the next
lemma.
\begin{lemma}\label{LTlocmod}
 $\BM^F_r$ is smooth of relative dimension $n-1$ over $O_E$. In fact,
 $\BM^F_r \otimes_{O_E} \bar k\simeq \BP^{n-1}_{\bar k}$. 
\end{lemma}
\begin{proof}The last statement implies the first. Indeed, then the
  general and the special fibre are smooth and irreducible of the same
  dimension $n-1$. The flatness of $\BM^{F}_r$ follows as in
  Corollary \ref{locmodflat}. This proves  smoothness.  

To study the geometric special fiber $\overline{\BM^F} = \BM^F \otimes_{O_E} \bar k$, we follow the method of section \ref{s:localmod}.
Let $W_0 = \Lambda \otimes_{O_{F^t}, \psi_0} \bar k$. Let $a_0 = \vert
A_{\psi_0} \vert$. Then, as in \eqref{cond12}, we have an identification 
\begin{equation*}
\begin{aligned}
\overline{\BM^F} (S) = \{ \CF_0 \subset W_0 \otimes_{\bar k} \CO_S \mid \pi\text{-stable direct summand,}\\
\text{ $\rank\,W_{0, S}/ \CF_0 = a_0n + 1$, and $1'), 2')$} \}
\end{aligned}
\end{equation*}
Here we have set $W_{0, S}=W_0\otimes_{\bar k}\CO_S$ and $1')$ and $2')$   are as in \eqref{cond12}, i.e.,
\begin{equation}
\begin{aligned}
1')  &\quad \pi^{(a_0+1)}\vert ({  W_{0, S}/ \CF_0})  &= 0 \, \\
2') &\wedge^{2} (\pi^{a_0}{\vert (W_{0, S}/\CF_0}))&= 0 .
\end{aligned}
\end{equation} 

Applying Lemma \ref{Eisen2l}, we obtain that for $\CF_0 \in
\overline{\BM^F} (S)$, there is a chain of inclusions,  
\begin{equation*}
 W^{a_0+1}_0 \otimes_{\bar k} \CO_S \subset \CF_0 \subset W^{a_0}_0
 \otimes_{\bar k} \CO_S \, , 
\end{equation*}
where $W^{a_0}_0 = \Im \pi^{a_0}$, resp. $W^{a_0 + 1}_0 = \Im \pi^{a_0
  + 1}$, is a $\bar k$-subspace of dimension $(e-a_0) n$,
resp. $(e-a_0 - 1) n$ of $W_0$. Associating to $\CF_0$ the submodule
$\CF_0 / (W^{a_0+1}_0\otimes_{\bar k} \CO_S)$ of $(W^{a_0 }_0 /
W^{a_0+1}_0) \otimes_{\bar k} \CO_S$, we obtain a  direct summand of
codimension one, i.e., an $S$-valued point of $\BP (W^{a_0}_0 /
W^{a_0+1}_0)$. Since this association is functorial 
and bijective, the last assertion of Lemma \ref{LTlocmod} follows.
\end{proof}

%------------------------------------------------------------------------
%------------------------------------------------------------------------
%-----------------------------------------------  

%------------------------------------------------------------------------------------------------------------------------------------------------------------------------------------------------

\bigskip
\obeylines
Mathematisches Institut der Universit\"at Bonn  
Endenicher Allee 60 
53115 Bonn, Germany
email: rapoport@math.uni-bonn.de

\bigskip
Fakult\"at f\"ur Mathematik
Universit\"at Bielefeld
Postfach 100131
33501 Bielefeld, Germany
email: zink@math.uni-bielefeld.de


\begin{thebibliography}{AB3}
\bibitem{A} T. Ahsendorf, \textsl{$\CO$-displays and $\pi$-divisible formal $\CO$-modules}, Thesis Bielefeld, 2011. 
\bibitem{AZ} T. Ahsendorf, Ch. Cheng, Th. Zink, \textsl{$\CO$-displays and $\pi$-divisible formal $\CO$-modules}, J. of Algebra {\bf 457} (2016), 129--193. 
\bibitem{BC} J.-F. Boutot and H. Carayol, \textsl{ Uniformisation $p$-adique des courbes de Shimura: les th\'eor\`emes
de Cerednik et de Drinfeld},  in: Courbes modulaires et courbes de Shimura, Ast\'erisque {\bf 196--197} (1991),  45--158. 
\bibitem{Ch} G. Chenevier, \textsl{The $p$-adic analytic space of pseudocharacters of a profinite group and pseudorepresentations over arbitrary rings}, in: Automorphic forms and Galois representations. Vol. 1, 221--285, London Math. Soc. Lecture Note Ser., {\bf 414}, Cambridge Univ. Press, Cambridge, 2014


\bibitem{D}  V. G. Drinfeld, \textsl{Coverings of p-adic symmetric regions}, Funct. Anal. Appl. {\bf 10} (1977), 29--40.

\bibitem{F} L. Fargues, \textsl{Cohomologie des espaces de modules de groupes $p$-divisibles et correspondances de Langlands locales}, in: Vari\'et\'es de Shimura, espaces de Rapoport-Zink et correspondances de Langlands locales, Ast\'erisque {\bf 291} (2004), 1--199.

\bibitem{G} U. G\"ortz, \textsl{On the flatness of models of certain Shimura varieties of PEL type},
Math. Ann. {\bf 321} (2001), 689-727.
\bibitem{GW} U. G\"ortz, T.Wedhorn,  Algebraic geometry I. Schemes with examples and exercises. Advanced Lectures in Mathematics. Vieweg + Teubner, Wiesbaden, 2010. Vieweg Verlag, Wiesbaden, 2010.
\bibitem{H} Ph. Hartwig, \textsl{Kottwitz-Rapoport and $p$-rank strata in the reduction of Shimura varieties of PEL type},  Ann. Inst. Fourier (Grenoble) {\bf 65} (2015), no. 3, 1031--1103.

 \bibitem{K} R. E. Kottwitz, \textsl{ Isocrystals with additional structure. II}, Compositio Math. {\bf 109} (1997), no. 3, 255--339. 
 
\bibitem{KR} S. Kudla and M. Rapoport, \textsl{New cases of $p$-adic uniformization}, Ast\'erisque {\bf 370} (2015), 207--241. 


\bibitem{KR3} S. Kudla and M. Rapoport, \textsl{ An alternative description of the Drinfeld p-adic half-plane}, Annales de l'Institut Fourier {\bf 64} (2014),  1203--1228. 

\bibitem{KRZ} S. Kudla, M. Rapoport and Th. Zink, \textsl{ On Cherednik uniformization}, in preparation.

\bibitem{L} E. Lau, \textsl{Displays and formal p-divisible groups},  Invent. Math. {\bf 171} (2008), no. 3, 617--628.

\bibitem{Mi} A. Mihatsch, \textsl{Relative unitary RZ-spaces and the Arithmetic Fundamental Lemma}, arXiv:1611.06520 

\bibitem{Ne} J. Nekovar, \textsl{Level raising and anticyclotomic Selmer groups for Hilbert modular forms of weight two}, Canad. J. Math. {\bf 64} (2012), 588--668.
\bibitem{PR.I}
G. Pappas and M. Rapoport, 
\textsl{Local models in the ramified case. I: The EL-case}, 
J. Alg. Geom. {\bf 12} (2003), 107--145.

\bibitem{PRS}
G. Pappas, M. Rapoport and B. Smithling, 
\textsl{ Local models of Shimura varieties, I. Geometry and combinatorics}, in:
Handbook of moduli (eds. G. Farkas and I. Morrison), vol. III, 135--217,
Adv. Lect. in Math. {\bf 26}, International Press, 2013. 

\bibitem{PZ} G. Pappas and  X. Zhu, \textsl{Local models of Shimura varieties and a conjecture of Kottwitz},  Inventiones math. {\bf 194} (2013), 147--254.

\bibitem{RV} M. Rapoport and E. Viehmann, \textsl{Towards a theory of local Shimura varieties}, M\"unster J. of Math. {\bf 7} (2014), 273--326. 
\bibitem{RZ}
M. Rapoport and T. Zink,  Period spaces for $p$-divisible groups. Annals  
of Mathematics Studies, {\bf 141}, Princeton University Press,  
Princeton, 1996.
\bibitem{RSZ3} M. Rapoport, B. Smithling and W. Zhang, \textsl {Arithmetic diagonal cycles on unitary Shimura varieties}, in preparation. 
\bibitem{Ri} K. Ribet, \textsl{On modular representations of $\Gal(\ov \BQ/\BQ)$ arising from modular forms}, Invent. math. {\bf100} (1990), 431--476. 
\bibitem{SW} P. Scholze and J. Weinstein, \textsl{Moduli of $p$-divisible groups}, Cambridge J. Math. {\bf 1} (2013), 145-237. 


\bibitem{Z} Th. Zink, \textsl{The display of a formal $p$-divisible
    group}, in: Cohomologies $p$-adiques et applications
  arithm\'etiques, I. Ast\'erisque {\bf 278} (2002), 127--248. 

\end{thebibliography}
\end{document}